\documentclass[11pt,reqno,twoside]{amsart}

\usepackage{graphicx}
\usepackage{mathrsfs}
\usepackage{color}
\usepackage{comment}
\usepackage{amssymb}
\usepackage{esint}
\usepackage{amsmath}
\usepackage{amsthm}

\allowdisplaybreaks

\usepackage[margin = 1.2in]{geometry}
\usepackage[hyperindex,breaklinks]{hyperref}
\usepackage{cleveref}

\newcommand{\on}{\operatorname}

\newcommand{\divv}{\on{div}}

\newcommand{\spt}{\mathrm{spt}}
\newcommand{\reg}{\mathrm{reg}}
\newcommand{\sing}{\mathrm{sing}}
\newcommand{\hit}{\mathrm{hit}}
\newcommand{\merge}{\mathrm{merge}}
\DeclareMathOperator{\spread}{\mathrm{sp}}
\DeclareMathOperator{\strength}{\mathrm{str}}
\newcommand{\dee}{\partial}

\DeclareMathOperator{\sign}{\mathrm{sgn}}

\newtheorem{lemma}{Lemma}
\newtheorem{proposition}[lemma]{Proposition}

\newtheorem{theorem}[lemma]{Theorem}
\newtheorem{conjecture}[lemma]{Conjecture}

\theoremstyle{definition}

\newtheorem{definition}[lemma]{Definition}

\title[A Partial $\epsilon$-Regularity Theorem for ISGNs in $\mathbb{R}^n$]{A Partial $\epsilon$-Regularity Theorem for Integral Stationary Geodesic Networks in $\mathbb{R}^n$}
\author[Henry Bosch]{Henry Bosch}

\address{Department of Mathematics, Stanford University, Building 380, Stanford, CA 94305, USA}
\email{hbosch@stanford.edu}
\date{}

\begin{document}

\begin{abstract}
We prove a partial $\epsilon$-regularity theorem for sequences of integer multiplicity stationary geodesic networks in $\mathbb{R}^n$ converging to a multiplicity $k$ line.
\end{abstract}

\maketitle

\setcounter{tocdepth}{1}

\section{Introduction}

An \emph{integral stationary geodesic network} in a Riemannian $n$-manifold $(M,g)$ is a locally finite embedded graph whose edges are geodesic segments, rays, complete geodesics, or closed geodesics, each carrying a positive integer multiplicity, which satisfies a geometric balancing condition at each vertex: namely, the sum of unit tangents along each geodesic leaving the vertex, counted with multiplicity, is zero. 

While many examples of integral stationary geodesic networks have been constructed (see \cite{AlexShapes}, \cite{Nab2023}), bounds on their complexity in general dimension $n$ are scarce. Given an arbitrary sequence of integral stationary geodesic networks whose local mass is uniformly bounded above, one can extract a local subsequential limit, which is also an integral stationary geodesic network (see \cite{Allard1976}).

In this paper, we provide an argument which controls the nature of the convergence of a sequence of integral stationary geodesic networks to a multiplicity $k$ line in $\mathbb{R}^n$. This provides a model for degeneration of a sequence of integral stationary geodesic networks into the top stratum of their limit network. We call our theorem a \emph{partial $\epsilon$-regularity theorem} because it shows that an integral stationary geodesic network in $\mathbb{R}^n$ sufficiently close to a multiplicity $k$ line has controlled structure, but is not necessarily regular.

When $n=2$, a bound was recently constructed for the number of singular points of integral stationary geodesic networks in Riemannian surfaces \cite{frolov2026geodesicnetseuclideanplane}. For an integral stationary geodesic network in $\mathbb{R}^2$ sufficiently close to a multiplicity $k$ line, their Theorem 1 bounds the number of singular points by $(50k)^{8k^2}$. 
It appears that our upper bound on the number of singular points satisfies $C(n,k) \leq n^{\frac{3}{2}(k-1)}(5k)^{k^2}$. In particular, when $n=2$, $C(2,k) \leq (10k)^{k^2}$.

\subsection{Main Theorem}

\begin{theorem} \label{maintheoremfinal}
    For any $n \geq 2$ and $k \geq 1$, there exists $C = C(n,k)$ such that if $(X_i)$ is any sequence of integral stationary geodesic networks in $\mathbb{R}^n$ which converge weakly (as Radon measures) to a multiplicity $k$ line $\ell$, and $K \subset \mathbb{R}^n$ is compact, then for sufficiently large $i$ (depending on $K$ and the $X_i$), the number of singular points of $X_i$ in $K$ is at most $C(n,k).$
\end{theorem}

This represents a limiting case of the following conjecture, which follows from \cite{frolov2026geodesicnetseuclideanplane} in the case $n=2$.

\begin{conjecture}[Spencer's Bus Stop Problem]
Fix a Riemannian $n$-manifold $(M,g)$, $n \geq 2$, and let $U \subset\subset V \subset\subset M$ be open. For any $L \geq 0$, there is a $C = C(M,U,V,L)$ such that any integral stationary geodesic network $X \subset M$ whose total length counted with multiplicity in $V$ is at most $L$ has at most $C$ singular points in $U$.
\end{conjecture}

We remark that no similar theorem may hold for a network in $M$ with sufficiently small curvature. In the Euclidean plane, consider the graph of the function $\epsilon^3 \sin(\frac{x}{\epsilon})$, and take its union with the $x$-axis. Then the curvature away from singular points is bounded by $\epsilon$, but the number of singular points in $B_1(0)$ is
$\Theta(\frac{1}{\epsilon})$ as $\epsilon \to 0.$

\subsection{Paper Organization} In \Cref{sec:geodesicnetworks} we define integral stationary geodesic networks in Riemannian manifolds. In \Cref{sec:structuretheorem} we prove a structure theorem for integral stationary geodesic networks sufficiently close to a multiplicity $k$ line. In \Cref{sec:maintechnicaltheorem} we state the main technical theorem, which bounds the number of singularities of a stationary geodesic network close to a multiplicity $k$ line with controlled structure (as given by the structure theorem). Sections \ref{sec:conservedquantities}, \ref{sec:centering}, \ref{sec:colllisionpotentials} and \ref{sec:spreadsmerging} are devoted to setting up the proof of the main technical theorem. Finally, \Cref{sec:globalargument} proves the main technical theorem, and \Cref{sec:proofmaintheorem} brings everything together to prove the main theorem.

\subsection{Acknowledgements} I would like to thank Otis Chodosh for introducing me to this problem, and Alexander Nabutovsky, Alec Shelley and Spencer Dembner for valuable conversations.

\subsection{On the use of AI} The author made use of AI to review the paper and correct typos. The paper contains no AI-generated text.

\section{Geodesic Networks}
\label{sec:geodesicnetworks}

In this section let $(M,g)$ be a Riemannian manifold of dimension $n \geq 2$.

\begin{definition}
    A \emph{geodesic network} in $M$ is a closed subset $X \subset M$ which is locally a finite embedded graph, each of whose edges is a geodesic segment, ray, complete geodesic, or closed geodesic. In this paper, we represent geodesic networks as sets $X \subset M$, and we write $X = X_\sing \sqcup X_\reg$, where $X_\reg \subset X$ is the set of points $x \in X$ around which $X$ is locally a geodesic segment, and $X_\sing = X \setminus X_\reg.$
\end{definition}

\begin{definition}
    A \emph{geodesic network with multiplicity} in $M$ is a geodesic network $X \subset M$ along with a multiplicity function $\Theta: X_\reg \to \mathbb{R}_+$ which is locally constant on $X_\reg.$ That is, each connected component of $X_\reg$ is assigned a positive real multiplicity. We say that $X$ is \emph{integral} if $\Theta$ takes on only positive integer values on $X_\reg$. Any geodesic network without multiplicity may be considered as one with multiplicity by setting $\Theta = 1$ on $X_\reg$.
\end{definition}

\begin{definition}
We extend $\Theta$ to all of $M$ as follows. For $x \in X_\sing$, let $\Theta_1, \ldots, \Theta_m > 0$ be the multiplicities of the geodesics touching $x$. Then set $\Theta(x) = \frac{1}{2}\sum_{j=1}^m \Theta_j$. If $x \notin X$, then set $\Theta(x) = 0$. Then $\Theta$ is the usual notion of local density for rectifiable varifolds (see \cite{simon_gmt}, 1.3).
\end{definition}

Note that $\Theta$ may have a positive noninteger value at a singular point $x \in X_\sing$. However this value always lies in $\frac{1}{2}\mathbb{Z}$.

\begin{definition}
    Let $X \subset M$ be a geodesic network with multiplicity. The \emph{measure associated to} $X$, written as $\mu_X$, is the Borel measure $\Theta \, d\mathcal{H}^1$, where $\mathcal{H}^1$ is one-dimensional Hausdorff measure on $M$. Notice that $\spt \,\mu_X = X$. We will write $d\mu_X = dX.$
\end{definition}

If $X \subset M$ is a geodesic network, then by local finiteness, $\mu_X$ is a Radon measure on $M$.

\begin{definition}
    Let $X \subset M$ be a geodesic network with multiplicity. We say that $X$ is \emph{stationary} if the following holds. Given any $x \in X_\sing$, let $s_1,\ldots,s_m \in T_xM$ be the unit tangent vectors pointing along the geodesics leaving $x$, and $\Theta_1,\ldots,\Theta_m$ be the corresponding multiplicities of the geodesics. Then 

    \[\sum_{i=1}^m \Theta_i \, s_i = 0.\]
\end{definition}

Stationary geodesic networks are critical points of the length functional on a suitable space of generalized curves on $M$. For more details, see \cite{Allard1976}.

For convenience we state a first variation formula for stationary geodesic networks. The proof is a standard cutoff argument, so we omit it. 

\begin{proposition}[First Variation] \label{first-variation}
Let $X \subset M$ be a stationary geodesic network with multiplicity. Let $U \subset\subset M$ be open, with the property that $X \cap U$ has finitely many ends, and each end has a unique limiting point in $\partial U$ ($U$ geodesically convex would suffice).
Let $Z$ be a $C^1$ vector field on $M$. Let $E$ be the set of ends of $X \cap U$. Each end $e \in E$ has a representative which is a geodesic segment $\gamma$ which limits to $\dee U$. For each $e \in E$ let $x_e \in \partial U$ be the limiting point in $\partial U$ of $\gamma$, let $\eta_e \in T_{x_e}M$ be the limiting unit tangent along $\gamma$ pointing outward from $U$, and let $\Theta(e) > 0$ be the multiplicity of $X$ along $\gamma$. Then
\[\int_U \divv_X Z \, dX = \sum_{e \in E} \Theta(e) \langle \eta_e, Z_{x_e}\rangle.\]
(Here for $x \in X$, $(\divv_X Z)|_x = \langle\eta, \nabla_\eta Z\rangle$, where $\eta \in T_xM$ is a choice of unit tangent to $X$ at $x$.)

\end{proposition}

\begin{definition}
    Let $X \subset M$ be a geodesic network with multiplicity. If $X$ is both integral and stationary we call it an \emph{integral stationary geodesic network}, or by abbreviation, an ISGN.
\end{definition}

\section{A Structure Theorem}
\label{sec:structuretheorem}

Fix $n \geq 2$. In this section, we prove a structure theorem for ISGNs sufficiently close to a multiplicity $k$ line in $\mathbb{R}^n$. In what follows, we let $E_1,\ldots,E_n$ be the standard orthonormal basis of $\mathbb{R}^n$ and of $T_x \mathbb{R}^n$ for any $x \in \mathbb{R}^n$. Let $x = (x_1,\ldots,x_n)$ be the standard coordinate on $\mathbb{R}^n$. Let $\pi = x_1$ be projection onto the $E_1$-axis.

\begin{definition}
    Let $B_r(0)$ denote the open ball of radius $r$ about the origin in $\mathbb{R}^{n-1}$. Suppose $A \subset \mathbb{R}$ is any subset. The \emph{tube} around $A$ of radius $r$ is the subset 
    \[T_{A,r} = A \times B_r(0) \subset \mathbb{R}^n.\]
    If $A$ is open, then $T_{A,r}$ is open. If $A = \{t\}$, write $T_{A,r} = T_{t,r}$.
\end{definition}

Now we can state the structure theorem for integral stationary geodesic networks in $\mathbb{R}^n$.

\begin{theorem}[Structure Theorem] \label{structure-theorem}
    Fix $k \geq 1$. For $i = 1,2,\ldots$ let $X_i$ be an ISGN in $\mathbb{R}^n$, and suppose that $X_i$ weakly converges to the multiplicity $k$ $E_1$-axis (as Radon measures on $\mathbb{R}^n$). Fix an open interval $I \subset\subset \mathbb{R}$ of finite length. Fix $0 < r < R$ and $\theta > 0$. Then there is an $N = N(n,k,(X_i), I, R, r, \theta)$ such that for all $i \geq N$,
    \begin{enumerate}
        \item $X_i \cap T_{I,R} \subset T_{I,r}$.
        \item At any regular point of $X_i \cap T_{I,R}$, any tangent vector to $X_i$ there makes an angle at most $\theta$ with the $E_1$-axis.
        \item At any singular point $x \in X_i \cap T_{I,R}$, the sum of the multiplicities of geodesic segments/rays entering $x$ from the left and leaving $x$ from the right (in the $E_1$ direction) are equal.
        \item For any $t \in I,$ $|X_i \cap T_{t,R}| \leq k$, and 
        \[\sum_{x \in X_i \cap T_{t,R}} \Theta_{X_i}(x) = k.\]
    \end{enumerate}
\end{theorem}

To prove the structure theorem, we need a series of lemmas. First, we prove that weak convergence implies local Hausdorff convergence.

\begin{lemma}[Local Hausdorff Convergence to the Line] \label{hausdorff-convergence}
    Suppose that $(X_i)$ is a sequence of ISGNs in $\mathbb{R}^n$ converging weakly to the multiplicity $k$ $E_1$-axis. Let $L \subset \mathbb{R}^n$ denote the $E_1$ axis, and let $K \subset \mathbb{R}^n$ be any compact set. Then
    \[\sup \, \{d(x,L): x \in K \cap X_i\} \to 0.\]
\end{lemma}

\begin{proof}
Assume not. After passing to a subsequence, there exists an $\epsilon > 0$ and $x_i \in X_i \cap K$ which satisfy $d(x_i,L) \geq 2\epsilon.$ By the monotonicity formula (see \cite{simon_gmt}, 4.3), we have $\mu_{X_i}(B_\epsilon(x_i)) \geq 2\epsilon$. Let $K_\epsilon = \{x \in \mathbb{R}^n: d(x,K) < \epsilon\}$. Then for any $x \in K$, $B_\epsilon(x) \subset K_\epsilon$.

Let $f \in C_c(\mathbb{R}^n)$ satisfy $f \geq 0$, $f(x) = 1$ if $x \in K_\epsilon$ and $d(x,L) \geq \epsilon$, and $f(x) = 0$ if $d(x,L) \leq \frac{\epsilon}{2}$. Then for any $i$,
\[\int_{\mathbb{R}^n} f \, dX_i \geq 2\epsilon,\]
but if $\mu$ is the measure associated to the multiplicity $k$-$E_1$ axis $L$,
\[\int_{\mathbb{R}^n} f \, d\mu = 0,\]
contradicting the weak convergence.
\end{proof}

Next we need a lemma that controls the integral tilt of an ISGN in terms of its distance to the $E_1$-axis.

\begin{definition}
    Let $\eta$ be a unit vector in $\mathbb{R}^n$. Its \emph{tilt} from the $E_1$-axis is the quantity $\sqrt{1 - \langle \eta,E_1\rangle^2}.$
\end{definition}

The tilt lies in $[0,1]$, is $0$ if $\eta$ is parallel to $E_1$ and is 1 if $\eta$ is perpendicular to $E_1$. Note that if $\phi = \angle(\eta, E_1)$ then the tilt is $\sqrt{1 - \langle\eta,E_1\rangle^2} = \sqrt{1 - \cos^2 \phi} = |\sin \phi|.$

\begin{lemma}[$L^2$ Tilt Control] \label{tilt-control}
Let $X$ be an ISGN in $\mathbb{R}^n$, and fix an open bounded $U \subset\subset \mathbb{R}^n$ which satisfies the hypotheses of Proposition \ref{first-variation}. As in that proposition, let $E$ be the set of ends of $X \cap U$, and for each $e \in E$, we have a point $x_e \in \partial U$, an outward pointing unit tangent $\eta_e \in T_{x_e}\mathbb{R}^n$, and a multiplicity $\Theta(e)$. For any $x \in \mathbb{R}^n$ let $r(x) = \sqrt{x_2^2 + \cdots + x_n^2}.$ Then
\[\int_U (1 - \langle \eta,E_1\rangle^2) \, dX \leq \sum_{e \in E} \Theta(e) \, r(x_e).\]
\end{lemma}

\begin{proof}
    Let $\eta$ be a choice of unit tangent on $X_\reg \cap U$ which is locally parallel on $X_\reg \cap U$. Note that $\eta$ may not agree with our choice of $\eta_e$ on the ends of $X$.

    Fix $j \in \{2,\ldots,n\}$ and consider the vector field $Z = x_jE_j$ on $\mathbb{R}^n$. Observe that on $X_\reg$, 
    \[\divv_X Z = \langle \eta, \nabla_\eta Z\rangle = \langle \eta, \eta[x_j] E_j + x_j \nabla_\eta E_j\rangle.\]
    But $\nabla_\eta E_j = 0$ and $\eta[x_j] = \langle \eta,E_j\rangle$ so 
    \[\divv_X Z = \langle \eta,E_j\rangle^2.\]
    Since $\eta$ is unique up to a sign, this is independent of our choice of $\eta$. Applying Proposition \ref{first-variation}, we obtain that 
    \[\int_U \langle \eta,E_j\rangle^2 \, dX = \sum_{e \in E} \Theta(e) (x_e)_j \langle \eta_e,E_j\rangle.\]
    Summing over $j$, and since $1 - \langle \eta,E_1\rangle^2 = \sum_{j=2}^n \langle \eta,E_j\rangle^2,$ we have 
    \[\int_U (1 - \langle \eta,E_1\rangle^2) \, dX = \sum_{e \in E} \Theta(e) \sum_{j=2}^n (x_e)_j \langle \eta_e,E_j\rangle \leq \sum_{e \in E} \Theta(e) \left|\sum_{j=2}^n (x_e)_j \langle \eta_e,E_j\rangle\right| \]
    \[\leq \sum_{e \in E} \Theta(e) \sqrt{\sum_{j=2}^n ((x_e)_j)^2\sum_{j=2}^n \langle \eta_e, E_j\rangle^2} \leq \sum_{e \in E} \Theta(e) r(x_e)\]
    and we are done.
\end{proof}

Next, let's guarantee that an ISGN which is contained in a tube and which only exits on the ends must span the whole tube.

\begin{lemma} \label{spans-tube}
Let $X$ be an ISGN in $\mathbb{R}^n$, $I \subset \subset \mathbb{R}$ an open interval of finite length, and $r > 0$. Suppose $X \cap T_{I,r} \neq \varnothing$ and $X \cap \partial T_{I,r} \subset T_{\partial I, r}$. Then, for each $t \in I$, $X \cap T_{t,r} \neq \varnothing$.  
\end{lemma}
\begin{proof}
    Suppose not. Since $I$ is connected, it suffices to show that the set $\pi(X \cap T_{I,r}) \subset I$ is both open and closed relative to $I$.
    
First, we show it is open. Fix $t \in \pi(X \cap T_{I,r})$, and let $Z$ be minimal $(X \cap T_{I,r})$-clopen set of $X \cap T_{I,r}$ containing $X \cap T_{t,r}$. If $Z \subset T_{t,r}$, then $Z$ must exit the tube $T_{I,r}$ along the neck at time $t$, contradiction. Thus there exists an $x \in Z$ at which there is a geodesic in $X$ leaving $x$ going positively or negatively in the $E_1$-direction. By the balancing condition, it follows that $X$ extends in both directions, which shows that $t$ lies in the interior of $\pi(X \cap T_{I,r})$.

Now we show it is closed. Fix $t_i \in \pi(X \cap T_{I,r})$ with $t_i \to t \in I$. Then select $x_i \in X \cap T_{t_i,r}$. Passing to a subsequence, $x_i \to x$, where $x \in \overline{T_{I,r}}$. But $\pi(x) = t \in I$ and $x \in X$. If $x \notin X \cap T_{I,r}$, then we must have $r(x)=r$, contradicting our assumption that $X \cap \partial T_{I,r} \subset T_{\partial I, r}$.
\end{proof}

Next, given an $L^2$ bound on the tilt, we want to find times which are regular, have small tilt, and have bounded total multiplicity. 

\begin{lemma} \label{find-point}
    Let $X$ be an ISGN in $\mathbb{R}^n$ and suppose that $I \subset\subset \mathbb{R}$ is open and bounded (not necessarily connected) with finite total length $L > 0$, and let $r > 0$. Suppose that for all $t \in I$, $X \cap T_{t,r}$ is nonempty. Further suppose that $\delta,C > 0$ and that we have an $L^2$ tilt bound
    \[\int_{T_{I,r}} (1 - \langle \eta,E_1\rangle^2) \, dX \leq \delta L\]
    as well as a mass bound
    \[\mu_X(T_{I,r}) \leq CL.\]
    Then there exists a $t^* \in I$ such that $X \cap T_{t^*,r}$ contains no singular points of $X$, and satisfies
    \[\sum_{x \in X \cap T_{t^*,r}} \Theta(x) \leq 3C\]
    and
    \[\max_{x \in X \cap T_{t^*,r}} \, (1 - \langle \eta_x, E_1\rangle^2) \leq 3\delta.\]
\end{lemma}

\begin{proof}
    For any $t \in I$ let
    \[m(t) = \sum_{x \in X \cap T_{t,r}} \Theta(x).\]
    By our assumption, we have $m(t) \geq 1$ for all $t \in I$. Note that if $t \in I$ has $ \#(X \cap T_{t,r}) = \infty$, then $m(t) = \infty$; this corresponds to segments of $X$ which are perpendicular to $E_1$. By local finiteness, $m(t) = \infty$ for only finitely many $t \in I$.
    
    For any $t \in I$, let
    \[A(t) = \sup_{x \in X \cap T_{t,r}} (1 - \langle \eta_x,E_1\rangle^2)\]
    if $X \cap T_{t,r}$ contains no singular points, and arbitrarily at other $t$ (of which there are finitely many).
    Fix $\epsilon > 0$ and let $J_1 = \{t \in I: m(t) > \frac{C}{\epsilon}\}.$ Since projection to the $E_1$ axis decreases length, we have
    \[\frac{C}{\epsilon}|J_1| \leq \int_{I} m(t) \, dt \leq \mu_X(T_{I,r}) \leq CL,\]
    so $|J_1| \leq \epsilon L$. Next let $J_2 = \{t \in I: A(t) > \frac{\delta}{\epsilon}\}.$ Similarly,
    \[\frac{\delta}{\epsilon}|J_2| \leq \int_{I} A(t) \, dt \leq \int_{T_{I,r}} (1 - \langle \eta,E_1\rangle^2) \, dX \leq \delta L,\]
    so $|J_2| \leq \epsilon L$. Taking $\epsilon = \frac{1}{3}$ we see that $|I \setminus J_1 \setminus J_2| \geq (1-\frac{2}{3})L > 0$. Since $X$ has finitely many singular points in $T_{I,r}$, take $t^* \in I \setminus J_1 \setminus J_2$ such that $X \cap T_{t^*,r}$ contains no singular points of $X$. Then $m(t^*) \leq 3C $ and $A(t^*) \leq 3\delta$, so we're done.
\end{proof}

The next lemma says that if we control the tilt on endpoints of a tube, then we control the tilt everywhere within the tube.

\begin{lemma} \label{brussel-sprout-lemma}
    Let $I = (t_1,t_2) \subset\subset \mathbb{R}$ be a bounded open interval, $r > 0$, $\delta > 0$, and $X \subset \mathbb{R}^n$ an ISGN which has the property that $X \cap \partial (T_{I,r}) \subset T_{\partial I, r}$. That is, $X$ exits the tube $T_{I,r}$ only ``along the ends''. As in Proposition \ref{first-variation} let $E$ be the set of ends of $X \cap T_{I,r}$. Suppose that for each $e \in E$, $r(\eta_e) \leq \delta$ (recall that $r(\eta) = \sqrt{\eta_2^2 + \cdots + \eta_n^2} = \sqrt{1 - \eta_1^2}$ is the tilt). Then at any regular point of $X \cap T_{I,r}$ any unit tangent $\eta$ to $X$ there has bounded tilt 
    \[r(\eta) \leq \delta \sum_{e \in E} \Theta(e).\]
\end{lemma}

\begin{proof}
Consider any $y \in X_\reg \cap T_{I,r}$, and let $\eta$ be a unit tangent to $X$ there at $y$. If $r(\eta) = 0$ there is nothing to prove, so assume that $r(\eta) > 0$. By a rotation of $\mathbb{R}^n$ which preserves $E_1$ (which leaves tilts invariant as well), we may assume that $\eta = (\eta_1,\eta_2,0,\ldots,0)$ where $\eta_2 > 0$. 

Let $R = \{x = (x_1,\ldots,x_n) \in T_{I,r}: x_2 < y_2\}$. Then $R$ is open and bounded in $\mathbb{R}^n$, so we can apply the first variation formula. Let $G$ be the set of ends of $X \cap R$. Then we can write $G = G_1 \sqcup G_2$, where $G_1 = \{e \in G: x_e \in T_{I,r}\}$ are ends whose points lie in the interior of $T_{I,r}$, and $G_2 = G \setminus G_1$. 
 Any $e \in G_1$ must exit $R$ along the face $\{x_2=y_2\} \cap T_{I,r} \subset \partial R$, so it must have $\langle \eta_e,E_2\rangle > 0$. Any $e \in G_2$ has $x_e \in \partial T_{I,r}$, so $x_e \in T_{\partial I,r}$ by our assumption, and we know that $r(\eta_e) \leq \delta$.
 
 Apply Proposition \ref{first-variation} with constant vector field $Z = E_2$.
We obtain that
\[0 = \sum_{e \in G} \Theta(e) \langle \eta_e,E_2\rangle = \sum_{e \in G_1} \Theta(e) \langle \eta_e ,E_2\rangle + \sum_{e \in G_2} \Theta(e) \langle \eta_e,E_2\rangle \geq \langle \eta_y,E_2 \rangle + \sum_{e \in G_2} \Theta(e)\langle \eta_e,E_2 \rangle. \]
Therefore
\[r(\eta_y) = \langle \eta_y,E_2\rangle \leq \sum_{e \in G_2} \Theta(e) | \langle \eta_e,E_2\rangle| \leq \delta \sum_{e \in E} \Theta(e) \]
so we are done.
\end{proof}

The next two lemmas show that if the angular deviation of $X$ is controlled, then the total multiplicity left of a singular point and right of a singular point must be equal. We choose to work with angles from now on instead of tilts, at the cost of a few constant factors.

\begin{lemma} \label{two-sides}
Let $U \subset\subset \mathbb{R}^n$ be open and let $X$ be an ISGN in $U$ with exactly one singular point $x$. Suppose further that no tangent to $X$ is perpendicular to $E_1$. Thus $X$ consists of a collection of geodesic segments $\mathcal{F}$ that enter $x$ from the left (in the $E_1$ direction) and a collection of geodesic segments $\mathcal{F}'$ that leave $x$ from the right. For a geodesic segment $I \in \mathcal{F} \cup \mathcal{F}'$, let $s_I \in S^{n-1}$ be the unique unit tangent vector to $I$ which points positively in the $E_1$ direction, let $\Theta(I)$ be the multiplicity of $I$, and let $M = \sum_{I \in \mathcal{F}} \Theta(I)$, $M' = \sum_{I' \in \mathcal{F}'} \Theta(I').$ If $M > M'$, then there exists $I \in \mathcal{F}$ such that $\angle(s_I,E_1) \geq \sqrt{\frac{2}{M}}$. Similarly if $M < M'$ then there exists $I' \in \mathcal{F}'$ such that $\angle(s_{I'},E_1) \geq \sqrt{\frac{2}{M'}}$.
    
\end{lemma}

\begin{proof}
   The stationarity condition asserts that
\[\sum_{I \in \mathcal{F}} \Theta(I) s_I = \sum_{I' \in \mathcal{F}'} \Theta(I') s_{I'}.\]
Without loss of generality assume that $M > M'$. Since $\langle s_{I},E_1\rangle \in (0,1]$ for all $I$, we have that
\[M' \geq \sum_{I' \in \mathcal{F}'} \Theta (I') \langle s_{I'} ,E_1\rangle = \sum_{I \in \mathcal{F}} \Theta(I)\langle s_{I},E_1\rangle.\]
Rearranging,
\[\sum_{I \in \mathcal{F}} \frac{\Theta(I)}{M}\langle s_{I},E_1\rangle \leq \frac{M'}{M}.\]
By properties of weighted averages, there exists $I \in \mathcal{F}$ such that
\[\langle s_I,E_1\rangle \leq \frac{M'}{M}.\]
Since $\cos \varphi \geq 1-\frac{\varphi^2}{2}$ for $\varphi \in [0,\frac{\pi}{2}]$, we get 
\[\frac{M'}{M} \geq \langle s_I,E_1\rangle \geq 1 - \frac{\angle(s_I,E_1)^2}{2},\]
so
\[\angle(s_I,E_1) \geq \sqrt{\frac{2(M-M')}{M}} \geq \sqrt{\frac{2}{M}}\]
as desired.
\end{proof}

\begin{lemma} \label{same-mult}
    Consider the same setup as in Lemma \ref{two-sides} but also assume that every unit tangent at a regular point of $X$ makes an angle at most $\theta$ with the $E_1$-axis. If $M \neq M'$, then there exists an $I \in \mathcal{F} \cup \mathcal{F}'$ such that 
    \[\angle(s_I,E_1) \geq \sqrt{\frac{2\cos \theta}{\min(M,M')}}.\]
\end{lemma}
\begin{proof}
    For any $I \in \mathcal{F} \cup \mathcal{F}'$, we have $\langle s_I,E_1\rangle \geq \cos \theta$. Without loss of generality, assume $M > M'$. Then we have
    \[M' \geq \sum_{I' \in \mathcal{F}'} \Theta(I') \langle s_{I'},E_1\rangle = \sum_{I \in \mathcal{F}} \Theta(I) \langle s_I,E_1\rangle \geq \cos \theta\sum_{I \in \mathcal{F}} \Theta(I) = M \cos \theta.\]

Thus $M \leq \frac{M'}{\cos \theta}$. From the conclusion of Lemma \ref{two-sides}, it follows that there exists $I \in \mathcal{F} \cup \mathcal{F}'$ such that 
\[\angle(s_I,E_1) \geq \sqrt{\frac{2}{M}} \geq \sqrt{\frac{2\cos \theta}{M'}}.\]
\end{proof}

Now we have enough to prove Theorem \ref{structure-theorem}.

\begin{proof}[Proof (of Theorem \ref{structure-theorem})] Fix $\epsilon \in (0,\frac{1}{4})$ and let $I' \subset \mathbb{R}$ be an $\epsilon$-neighborhood of $I$. That is, if $I = (a,b)$ then $I' = (a-\epsilon,b+\epsilon)$. Notice that given $r,\theta > 0$, we are free to choose $r', \theta'$ and prove the theorem for them, which implies the theorem for the given $r,\theta$. Thus we allow ourselves freely to decrease $r$ and $\theta$.

Decrease $\theta$ if necessary to ensure that
\[\theta < \sqrt{\frac{2\cos \theta}{3\frac{1+\epsilon}{1-4\epsilon}k}}\]
and that
\[ \cos \theta > \frac{k-1}{k}. \]
Then, decrease $r$ if necessary to ensure that 
\[12\sqrt{6} \left(\frac{1+\epsilon}{1-4\epsilon}\right)^{3/2}  \frac{k^{3/2}\sqrt{r}}{\epsilon} \leq \theta.\] 
Let $K = \overline{T_{I',R}}$. By local Hausdorff convergence of $X_i$ to the $E_1$-axis Lemma \ref{hausdorff-convergence} using $K$, we can choose $N$ such that for all $i \geq N$, $X_i \cap T_{I',R} \subset T_{I',r}$. 

The set $I' \setminus I$ has two connected components $(a-\epsilon,a]$ and $[b,b+\epsilon)$. Let $J_1 = (a-\epsilon,a)$ and $J_2 = (b,b+\epsilon)$ be their respective interiors. 

Then, by weak convergence of $X_i$ to the multiplicity $k$ $E_1$-axis, increase $N$ if necessary such that for all $i \geq N$ and $c \in \{1,2\}$,
\[(1-\epsilon)k |J_c| \leq \mu_{X_i}(T_{J_c,R}) \leq (1+\epsilon)k |J_c|.\]
In particular, this shows that $X_i \cap T_{I',r}$ is nonempty. Since $X_i \cap T_{I',R} \subset T_{I',r}$, $X_i$ cannot leave $T_{I',r}$ along its neck, so
by Lemma \ref{spans-tube}, for every $t \in I'$, $X_i \cap T_{t,r} \neq \varnothing$.

For any $t \in I'$, let $m_i(t) = \sum_{x \in X_i \cap T_{t,R}} \Theta(x)$. We make the following claim.

\ 

\textbf{Claim.} For all $i \geq N$, there exists an open interval $I''$ with $I \subset I'' \subset I'$ such that every tangent at a regular point of $X_i \cap T_{I'',R}$ makes an angle at most $\theta$ with the $E_1$-axis. Further, if $t^* \in \partial I''$ then
$m_i(t^*) \leq 3\frac{1+\epsilon}{1-4\epsilon} k.$

\

\emph{Sketch of proof of Claim.} Fix $i$. For $J_c$, $c \in \{1,2\}$, do the following. Use weak convergence to the multiplicity $k$ line to find two points $t_c < t_c'$ in $J_c$ close to either end of $J_c$ such that the local multiplicities $m_i(t_c)$ and $m_i(t'_c)$ are comparable to $k$. Then, use Lemma \ref{tilt-control} to show that the $L^2$ integral of tilt of $X_i$ over $(t_c,t'_c)$ is small. Then, plug that bound, along with a mass bound of $X_i$ over $(t_c,t'_c)$, into Lemma \ref{find-point} to obtain a point $t_c^*$ in $(t_c,t'_c) \subset J_c$ over which there is bounded multiplicity and bounded tilt. Let $I'' = (t_1^*,t_2^*).$ Since we have bounded tilt and multiplicity over $\partial I''$, we can apply Lemma \ref{brussel-sprout-lemma} to bound tilt everywhere on $I''$, which is equivalent to bounding the angle with the $E_1$-axis. 

\

\emph{Proof of Claim.} Fix $i$. Everything below will be carried out for $c \in \{1,2\}$. 
First we construct the points $t_c < t_c'$ in $J_c$. Observe that
\[(1+\epsilon)k|J_c| \geq \mu_{X_i}(T_{J_c,R}) \geq \int_{J_c} m_i(t) \, dt.\]
If we let
\[H_c = \left\{t \in J_c: m_i(t) > \frac{1+\epsilon}{\epsilon}k\right\},\]
then
\[(1+\epsilon)k|J_c| \geq \int_{J_c} m_i(t) \, dt \geq \frac{(1+\epsilon)k}{\epsilon}|H_c|,\]
so $|H_c| \leq \epsilon |J_c|.$ Thus $H_c$ cannot fill up either of the length $2\epsilon |J_c|$ intervals closest to either end of $J_c$. Let $t_c,t_c' \in J_c \setminus H_c$ be on the left and right sides of $J_c$, each within $2\epsilon |J_c|$ of $\partial J_c$. Then $t_c < t_c'$, $|t'_c - t_c| \geq (1-4\epsilon)|J_c| = \epsilon(1-4\epsilon)$, and $m_i(t_c),m_i(t'_c) \leq \frac{1+\epsilon}{\epsilon}k.$

Let $J'_c = (t_c,t_c')$. Next, we apply Lemma \ref{tilt-control} to the set $U_c = T_{J'_c,R}$, and we conclude that
\[\int_{U_c} (1-\langle \eta,E_1\rangle^2) \, dX_i \leq \sum_{e \in E_c} \Theta(e) r(x_e) \leq r\sum_{e \in E_c}\Theta(e) \leq 2kr \frac{1+\epsilon}{\epsilon},\]
where $E_c$ is the set of ends of $X_i \cap U_c$. Since $|J'_c| \geq \epsilon(1-4\epsilon),$ we have 
\[\int_{U_c} (1-\langle \eta,E_1\rangle^2) \, dX_i \leq 2kr\frac{1+\epsilon}{\epsilon} \leq 2kr \frac{1+\epsilon}{\epsilon^2(1-4\epsilon)} |J'_c|.\]

Now we can apply Lemma \ref{find-point} to the interval $J'_c$ with $\delta = \frac{2kr(1+\epsilon)}{\epsilon^2(1-4\epsilon)}$. Since
\[\mu_{X_i}(T_{J'_c,r}) \leq \mu_{X_i}(T_{J_c,r}) \leq (1+\epsilon)k |J_c| = \epsilon(1+\epsilon)k \leq \frac{1+\epsilon}{1-4\epsilon}k |J'_c|,\]
we can take $C = \frac{1+\epsilon}{1-4\epsilon}k$ in the mass bound. Therefore, by Lemma \ref{find-point}, there exists a $t^*_c \in J'_c \subset J_c$ such that the sets $X_i \cap T_{t_c^*,r}$ for $c=1,2$ have no singular points of $X_i$ and satisfy

\[\sum_{x \in X_i \cap T_{t_c^*,r}
} \Theta(x) \leq 3\frac{1+\epsilon}{1-4\epsilon} k\]
and
\[\max_{x \in X_i \cap T_{t_c^*,r}} (1-\langle \eta_x,E_1\rangle^2) \leq \frac{6(1+\epsilon)}{(1-4\epsilon)}\frac{kr}{\epsilon^2}.\]

Let $I'' = (t_1^*,t_2^*)$. Apply Lemma \ref{brussel-sprout-lemma}, where \[ \delta = \sqrt{\frac{6(1+\epsilon)}{(1-4\epsilon)}}\frac{\sqrt{kr}}{\epsilon} \]
is the upper bound on tilt on the endpoints. We conclude that every tilt at a regular point of $X_i \cap T_{I'',R}$ is at most

\[6\sqrt{6} \left(\frac{1+\epsilon}{1-4\epsilon}\right)^{3/2}  \frac{k^{3/2}\sqrt{r}}{\epsilon}. \]

Since $|\phi| \leq 2|\sin \phi|$ for any $\phi \in [0,\frac{\pi}{2}]$, we can bound the angle by two times the tilt. Therefore at every regular point of $X_i \cap T_{I'',R}$, the angle against the $E_1$-axis is at most

\[12\sqrt{6} \left(\frac{1+\epsilon}{1-4\epsilon}\right)^{3/2}  \frac{k^{3/2}\sqrt{r}}{\epsilon} \]
which is at most $\theta$ by assumption. That completes the proof of the claim.

\

Let's continue with the proof of Theorem \ref{structure-theorem}. To prove (3), suppose for contradiction it is false, and let $x \in X_i \cap T_{I'',R}$ be a singular point failing the condition in (3) which is minimal with respect to the $x_1$-coordinate (such a point exists, since by local finiteness, $X_i$ has finitely many singular points in $T_{I'',R}$). By minimality, the total multiplicity $M$ coming into $x$ from the left is at most $m_i(t_1^*) \leq 3\frac{1+\epsilon}{1-4\epsilon}k$. This contradicts Lemma \ref{same-mult}, by our assumption that $\theta < \sqrt{\frac{2\cos \theta}{3\frac{1+\epsilon}{1-4\epsilon}k}}$.

\ 

Finally we move to prove (4). From the previous point, we know that $m_i(t)$ is constant on $I''$, so it is constant on $I$. Since all tangents are within an angle $\theta$ with the $E_1$-axis, it follows that
\[m_i |I| \leq \mu_{X_i}(T_{I,R}) \leq \frac{m_i|I|}{\cos \theta}.\]
Letting $a_i = \frac{1}{|I|}\mu_{X_i}(T_{I,R}),$
we have that
\[m_i \leq a_i \leq \frac{m_i}{\cos \theta}.\]
 By weak convergence we have
\[a_i|I| = \mu_{X_i}(T_{I,R}) \to k|I|,\]
so $a_i \to k$. 
Therefore 
\[\liminf_{i \to \infty}m_i \geq \liminf_{i \to \infty} a_i \cos \theta  = k \cos \theta \]
and
\[\limsup_{i \to \infty}m_i \leq \limsup_{i \to \infty} a_i = k.\]
But the $m_i$ are integers. Since we assumed that $\cos \theta > \frac{k-1}{k}$, we get that 
$\liminf_{i \to \infty} m_i > k-1,$
so
$\liminf_{i \to \infty} m_i = k,$ so $\lim_{i \to \infty}m_i = k.$ Increase $N$ again if necessary such that for all $i \geq N$, $m_i = k$, and we are done.
\end{proof}

\section{Main Technical Theorem}
\label{sec:maintechnicaltheorem}

In this section, we introduce the main technical theorem, which bounds the total number of singular points of an integral stationary geodesic network in $\mathbb{R}^n$ which has sufficiently controlled structure.

\begin{definition}
    Let $X \subset \mathbb{R}^n$ be an ISGN and let $\theta \in [0,\frac{\pi}{2})$. We say that $X$ is $\theta$-\emph{flat} if every geodesic or ray in $X$ makes an angle at most $\theta$ with the $E_1$-axis.
\end{definition}

\begin{definition}
    Let $X \subset \mathbb{R}^n$ be an ISGN, and let $k \geq 1$ be an integer. We say that $X$ is a $k$-\emph{cover} if for every $t \in \mathbb{R}$, $|X \cap \pi^{-1}(t)| \leq k$ and 
    \[\sum_{x \in X \cap \pi^{-1}(t)} \Theta(x) = k.\] Further, we require that at any singular point $x \in X_\sing$, the total multiplicity of the geodesics entering $x$ from one side (with respect to the $E_1$ direction) is equal to the total multiplicity of the geodesics leaving $x$ from the other side.
\end{definition}

This means that $X$ consists of $k$ geodesics, counted with multiplicity, which interact via $X_\sing$. Now we have enough to state our main technical theorem.

\begin{theorem} \label{maintechnicaltheorem}
    For each $n \geq 2$ and $k \geq 1$, there exists a $\theta = \theta(n,k)$ and $C = C(n,k)$ with the following property. Suppose that $X \subset \mathbb{R}^n$ is an ISGN which is a $k$-cover, $\theta$-flat, and has $|X_\sing|< \infty.$ Then $|X_\sing| \leq C.$
\end{theorem}

We remark that the hypothesis $|X_\sing| < \infty$ is not needed: if $X$ is a $k$-cover and $\theta(n,k)$-flat, we claim that $|X_\sing| < \infty$. Indeed, an ISGN $X \subset \mathbb{R}^n$ has finitely many singular points in any ball $U = B_R(0)$. One can then extend the geodesics from $X \cap U$ to rays, and these may intersect only finitely many times outside $U$. This yields an ISGN $X' \subset \mathbb{R}^n$ with finitely many singular points which agrees with $X$ on $U$. Observe that $X'$ is still $\theta(n,k)$-flat and a $k$-cover. Applying the theorem to $X'$, we see that $X'$ has at most $C(n,k)$ singular points, so $X$ has at most $C(n,k)$ singular points in $U$.
But the bound does not depend on $R$, so the total number of singular points of $X$ is bounded by $C(n,k)$, and is thus finite.

\section{Conserved Quantities}
\label{sec:conservedquantities}

In this section $X \subset \mathbb{R}^n$ is an ISGN which is a $k$-cover and $\theta$-flat, $\theta \in [0,\frac{\pi}{2})$, with $|X_\sing| < \infty.$

\begin{definition}
    For any $x \in X_\reg$, let $s_x \in \mathbb{R}^n$ be the unique unit vector tangent to $X$ at $x$ such that $\langle s_x,E_1\rangle > 0$.
\end{definition}

\begin{definition}
    For any $t \in \mathbb{R}$ let $X_t = X \cap \pi^{-1}(t)$. We call $X_t$ the \emph{time-slice} of $X$ at \emph{time} $t$.
    Since $X$ is a $k$-cover, for all $t$, we have $\sum_{x \in X_t} \Theta(x) = k.$
\end{definition}

\begin{definition}
    For any $A \subset \mathbb{R}$ let $X_A = X \cap \pi^{-1}(A)$. We call $X_A$ the \emph{time-slab} of $X$ over $A$.
\end{definition}

\begin{definition}
    Let $\tilde V = \pi(X_\sing)$, and let $\tilde U = \mathbb{R} \setminus \tilde V$. We call $\tilde V$ the set of \emph{singular times}, and $\tilde U$ the set of \emph{regular times}. Since $|X_\sing| < \infty$, $\tilde V$ is finite.
\end{definition}

\begin{definition}
    For any $t \in \tilde U$, let $W(t) = \sum_{x \in X_t} \Theta(x) \, s_x.$
\end{definition}

\begin{proposition} $W(t)$ is independent of $t$, so is an invariant of $X$, which we call $W$.
\end{proposition}

\begin{proof}
    Observe that $W(t)$ is locally constant on $\tilde U$, so it suffices to show that $W(t)$ does not jump across a point $t \in \tilde V$. Fixing this $t$, it suffices to show that the terms coming from geodesics interacting in one singular point $x \in X_t$ have equal contribution on either side of $t$. Suppose that we have geodesics coming into $x$ from the left (in the $E_1$ direction), with rightward-pointing unit tangents $s_1,\ldots,s_m$ and multiplicities $\Theta_1,\ldots,\Theta_m$, and suppose we have geodesics leaving $x$ from the right, with rightward-pointing unit tangents $s_1', \ldots, s_{m'}'$ and multiplicities $\Theta'_1,\ldots,\Theta'_{m'}$. By the stationarity condition, we have 
    \[\sum_{i=1}^m \Theta_i \, s_i = \sum_{i=1}^{m'} \Theta_i' \, s_i',\]
    which precisely says that the contribution to $W$ is conserved across $t$.
\end{proof}

Observe that for any $s \in S^{n-1}$, $\langle W,s\rangle \leq k$, and that equality is attained if and only if every tangent to $X$ is parallel to $s$, i.e. $X$ consists of $k$ parallel lines tangent to $s$, counted with multiplicity. Since $\langle s_x,E_1\rangle > 0$ for all $x \in X_\reg$, we have $\langle W,E_1\rangle > 0.$

\begin{definition}
    For any unit vector $s \in S^{n-1}$, we define the $s$-\emph{tilt} $A_s = \sqrt{2(k-\langle W,s\rangle)}.$
\end{definition}

\begin{lemma} \label{tilt-formula}
    For any $s \in S^{n-1}$ and $t \in \tilde U$, we have
    \[A_s^2 = \sum_{x \in X_t} \Theta(x) \, |s_x-s|^2.\]
\end{lemma}
\begin{proof}
    Since all $s_x$ and $s$ are length one, we have $|s_x-s|^2 = 2-2\langle s_x,s\rangle.$ Therefore for any $t \in \tilde U$,
    \[\sum_{x \in X_t} \Theta(x) \, |s_x-s|^2 = \sum_{x \in X_t} \Theta(x) \, (2 - 2\langle s_x,s\rangle) = 2k - 2 \langle W,s\rangle = A_s^2.\]
\end{proof}

\begin{definition}
    The \emph{central direction} of $X$ is $S = \frac{W}{|W|} \in S^{n-1}.$ 
\end{definition}

\begin{lemma}
    The central direction satisfies $\angle(S,E_1) \leq \theta.$
\end{lemma}
\begin{proof}
    Fix $t \in \tilde U$ and let $s_1,\ldots,s_k$ be the unit tangents to $X_t$, repeated with multiplicity. Then $W = \sum_{i=1}^k s_i$. Since $X$ is $\theta$-flat, $\langle s_i,E_1\rangle \geq \cos \theta$ for all $i$. Since $|W| \leq k$, we have 
    \[\langle S,E_1\rangle = \frac{1}{|W|} \sum_{i=1}^k \langle s_i,E_1\rangle \geq \frac{k \cos \theta}{|W|} \geq \cos \theta,\]
    so $\angle(S,E_1) \leq \theta.$
\end{proof}

Observe that if $\angle(S,E_1) = \theta$, then $|W| = k$, so $X$ consists of $k$ parallel lines, counted with multiplicity.

\begin{definition}
    Taking $s$ to be the central direction $S$, let $\delta := A_S$ be the \emph{central tilt} of $X$. From Lemma \ref{tilt-formula}, for all $t \in \tilde U$, we have
    \[\delta^2 = 2(k-|W|) = \sum_{x \in X_t} \Theta(x) \, |s_x-S|^2.\]
\end{definition}

\begin{lemma} \label{central-tilt-equivalence}
    The central tilt satisfies $0 \leq \delta \leq 2\sqrt{k} \, \theta$ and $\delta = 0$ if and only if $X$ consists of $k$ parallel lines, counted with multiplicity.
\end{lemma}

\begin{proof}
Take any $t \in \tilde U$. Then 
\[\delta^2 = \sum_{x \in X_t} \Theta(x) \, |s_x - S|^2. \]
Recall that for any $a,b \in S^{n-1}, |a-b| \leq \angle(a,b).$ Thus, for any $x \in X_t$,
\[|s_x-S| \leq \angle(s_x,S) \leq \angle(s_x,E_1) + \angle(E_1,S) \leq 2\theta\]
so 
\[\delta^2 = \sum_{x \in X_t} \Theta(x) \, |s_x - S|^2 \leq \sum_{x \in X_t} \Theta(x) \, (2\theta)^2 = 4k\theta^2.\]
Clearly $\delta \geq 0$, and equality holds if and only if all $s_x$ are equal to $S$, in which case $X$ consists of $k$ parallel lines, counted with multiplicity.
\end{proof}

\begin{lemma}
    Let $s$ be any rightward pointing unit tangent to $X$ at a regular point. Then $|s-S| \leq \delta$, and equivalently $\langle s,S\rangle \geq 1 - \frac{\delta^2}{2}.$
\end{lemma}
\begin{proof}
From Lemma \ref{tilt-formula} we know that $|s - S| \leq \delta$. The second part follows because
\[\delta^2 \geq |s-S|^2 = 2-2\langle s,S\rangle.\]
\end{proof}

Further, we have

\begin{lemma} \label{spread-formula}
    At any $t \in \tilde U$,
    \[\frac{1}{2} \sum_{x,y \in X_t} \Theta(x) \Theta(y) \, |s_x-s_y|^2 = k \delta^2- \frac{\delta^4}{4}.\]
\end{lemma}
\begin{proof}
    Fix $t \in \tilde U$. We choose to duplicate $x,y \in X_t$ according to their multiplicity. Then we have
    \[\sum_{x,y \in X_t} |s_x-s_y|^2 = \sum_{x,y \in X_t} (2 - 2\langle s_x,s_y\rangle)\]
    \[= 2k^2 - 2 \left\langle \sum_{x \in X_t} s_x, \sum_{y \in X_t}s_y\right\rangle = 2k^2 - 2|W|^2.\]
    Next, observe that
    \[k+|W| = 2k - (k-|W|) = 2k - \frac{\delta^2}{2}.\]
    Therefore
    \[k^2 - |W|^2 = (k+|W|)(k-|W|) = \left(2k - \frac{\delta^2}{2}\right)\frac{\delta^2}{2} = k \delta^2 - \frac{\delta^4}{4},\] which proves the formula.
\end{proof}

\section{Centering}
\label{sec:centering}

In this section, $X \subset \mathbb{R}^n$ is an ISGN which is a $k$-cover and $\theta$-flat, $\theta \in [0,\frac{\pi}{2})$, with $|X_\sing| < \infty.$ It will be necessary to center $X$, i.e. rotate it so that $S = E_1$. However it makes the argument simpler to perturb the centered $X$ slightly.

\begin{definition}
    We say that $X$ is \emph{centered} if $S = E_1$. For $\gamma \in (0,\frac{1}{4})$, we say that $X$ is $\gamma$-centered if $\angle( S,E_1) \leq \gamma \delta.$
\end{definition}

Then we have

\begin{lemma} \label{almost-centered-slope-bounds}
    Suppose $X$ is $\gamma$-centered. Let $s$ be a rightward pointing unit tangent to $X$ at any regular point. Then $|s-E_1| \leq (1+\gamma)\delta$. Equivalently, 
    \[\langle s,E_1\rangle \geq 1 - \frac{(1+\gamma)^2}{2}\delta^2.\]
    It follows that for any $j \geq 2$, $|\langle s,E_j\rangle| \leq (1+\gamma)\delta$. Since we have taken $\gamma < \frac{1}{4}$ we obtain the symbolically simpler inequalities
    \[\langle s,E_1\rangle \geq 1-\delta^2\]
    and
    \[|\langle s,E_j\rangle| \leq 2\delta\]
    for each $j \geq 2$.
\end{lemma}

\begin{proof}
    We have $|s-E_1| \leq |s-S|+|S-E_1| \leq \delta + \gamma \delta.$ Next, $\langle s,E_1\rangle =1 - \frac{1}{2}|s-E_1|^2 \geq 1 - \frac{(1+\gamma)^2}{2}\delta^2$. Finally, $|s_j| \leq \sqrt{(s_1-1)^2 + s_2^2 + \cdots + s_n^2} = |s-E_1| \leq (1+\gamma)\delta.$ The last two inequalities are numerology.
\end{proof}

\begin{definition}
A \emph{ghost-crossing} in $X$ is a triple $(x,x',j)$, where $x,x' \in X$, $x \neq x'$, and $j \in \{2,\ldots, n\}$, such that $x_1 = x_1'$ and $x_j = x_j'$. We consider the ghost-crossings $(x,x',j)$ and $(x',x,j)$ to be equivalent.

The \emph{direction} of the ghost-crossing is the index $j$. The \emph{time} of the ghost-crossing is $x_1=x_1'$. A ghost crossing is \emph{regular} if $x,x' \in X_\reg$, is \emph{half-singular} if one of $x,x'$ is singular and one is regular, and is \emph{singular} if $x,x' \in X_\sing$.

A ghost-crossing is \emph{degenerate} if there exist $s,s'$ rightward pointing unit tangents to $x,x'$ respectively (if the point is singular they can be any tangent just to the left or right of the singular point), such that
\[\frac{s_j}{s_1} = \frac{s'_j}{s'_1}.\]
\end{definition}

Observe that in the case $n=2$, there are no ghost-crossings.

\begin{definition}
    We say that $X$ is \emph{good} if it contains no degenerate (of any regularity-type), singular, nor half-singular ghost-crossings, and for each $t \in \mathbb{R}$, $X_t$ has at most one singular point.
\end{definition}

\begin{lemma} \label{bound-ghost}
    Suppose that $X$ is \emph{good}. Then $X$ has finitely many ghost-crossings (all of which are regular). In fact, the number of regular ghost-crossings is at most $(|X_{\sing}| + 1)(n-1){k \choose 2}$.
\end{lemma}

\begin{proof}
    Since $X$ is good, it suffices to bound regular ghost-crossings.
    Let $I$ be a connected component of $\tilde U$. We will show that the number of ghost-crossings whose times lie in the closure $\overline I$ and are in direction $j$ is at most ${k \choose 2}$. Since there are $n-1$ choices for $j$ and $|X_\sing| + 1$ choices for $I$, that will complete the proof.

    Observe that $X_{\overline{I}}$ is a union of $m \leq k$ closed segments or rays (if $I$ is unbounded), which may intersect at $X_\sing$ at times in $\partial I$. It suffices to show each pair of distinct segments has at most one regular ghost crossing in direction $j$.

    Let $\gamma,\gamma'$ be two such distinct closed geodesics.  Parameterize them in time as $x(t)$ and $x'(t)$, where $x_1(t) = x'_1(t) = t$. Let $s$ and $s'$ be the rightward pointing unit tangent vectors along $\gamma$ and $\gamma'$ respectively. Observe that for all $t \in \overline{I}$, $\frac{s_j}{s_1} = \frac{\partial x_j}{\partial t}$ and $\frac{s_j'}{s_1'} = \frac{\partial x_j'}{\partial t}.$ 
    
    Firstly suppose that $\frac{s_j}{s_1} = \frac{s'_j}{s'_1}$. Then $x_j-x_j'$ is a constant $c$ along $t \in \overline{I}.$ If $c \neq 0$, then there is no ghost-crossing between $\gamma$ and $\gamma'$ in direction $j$. Instead if $c = 0$, then there is a degenerate ghost-crossing between $\gamma$ and $\gamma'$ in direction $j$ at every time $t \in \overline{I}$,
    which is contrary to the assumption that $X$ is good. 

Secondly suppose instead that $\frac{s_j}{s_1} \neq \frac{s'_j}{s'_1}$. Then $x_j(t) - x'_j(t) = a + bt$ for some $a \in \mathbb{R}$ and $b \neq 0$. Thus there is at most one $t \in \overline{I}$ with $x_j(t) = x'_j(t)$, so there is at most one ghost crossing between $\gamma$ and $\gamma'$ in direction $j$ over $\overline{I}$.
\end{proof}

We wish to slightly rotate $X$ such that it becomes good. In order to do this, we need to define what we mean by a slight rotation and guarantee that a slight rotation of $X$ will still be quantitatively close to the multiplicity $k$ $E_1$-axis.

\ 

Let $T(n) \cong \mathbb{R}^n$ denote the Lie group of translations of $\mathbb{R}^n$. Let $SE(n)$ denote the special Euclidean group, i.e. the Lie group of all orientation-preserving rigid motions of $\mathbb{R}^n$. Recall that $SE(n)$ can be written as a semidirect product $SE(n) \cong SO(n) \ltimes T(n)$. That is, we have a split short exact sequence

\[1 \longrightarrow T(n) \longrightarrow SE(n) \longrightarrow SO(n) \longrightarrow 1.\]

The inclusion $T(n) \to SE(n)$ is clear since each translation is a rigid motion. The quotient map $SE(n) \to SO(n)$ may be viewed as follows.

Given $x \in \mathbb{R}^n$, $SE(n)$ acts on $T_x \mathbb{R}^n$ as follows. Given $\phi \in SE(n)$ and $v \in T_x\mathbb{R}^n$, $\phi(v) \in T_x \mathbb{R}^n$ is $\phi_* v$ parallel transported back to $x$ (the parallel transport map does not depend on the path due to the flatness of $\mathbb{R}^n$). Because this is a composition of local isometries, this action yields a homomorphism $R_x: SE(n) \to SO(T_x \mathbb{R}^n).$ Since any translation yields the identity map on $T_x\mathbb{R}^n$, this induces an isomorphism $SO(n) \cong SE(n)/T(n) \to SO(T_x\mathbb{R}^n),$ and for any $x \in \mathbb{R}^n$, we have a split exact sequence

\[1 \longrightarrow T(n) \longrightarrow SE(n)  \overset{R_x}\longrightarrow SO(T_x\mathbb{R}^n) \longrightarrow 1.\]

Let $L_x: SO(T_x\mathbb{R}^n) \to SE(n)$ be the map which takes $\varphi \in SO(T_x\mathbb{R}^n)$ to the unique rigid motion of $\mathbb{R}^n$ which fixes $x$ and whose differential at $x$ is $\varphi$. Then $R_x \circ L_x = \mathrm{Id}_{SO(T_x\mathbb{R}^n)}$, and this yields a splitting of the above exact sequence. 

For any $x,y \in \mathbb{R}^n$, the maps $R_x\phi$ and $R_y\phi$, viewed as $(1,1)$-tensors at $x$ and $y$, are parallel transports of each other. Therefore if we identify all the tangent spaces of $\mathbb{R}^n$ with each other via parallel transport, all $R_x \phi$ over $x \in \mathbb{R}^n$ agree, and we just call them $\phi_*$ by abuse of notation.

\begin{definition}
Given $\alpha > 0$, we let 
\[SE_\alpha(n) = \{\phi \in SE(n): \angle(E_1,\phi_*E_1) < \alpha\}.\] Note that $SE_\alpha(n)$ is not a subgroup of $SE(n)$.
\end{definition}

Then we have 

\begin{lemma} \label{rotation-preserves}
    Suppose that $\alpha \leq \frac{\pi}{2} - \theta$ and $\phi \in SE_\alpha(n)$. Then $\phi(X)$ is an ISGN in $\mathbb{R}^n$ which is $(\theta + \alpha)$-flat. Further, $\phi(X)$ is a $k$-cover.
\end{lemma}

\begin{proof}
    Clearly $\phi(X)$ is an ISGN in $\mathbb{R}^n$. We first claim it is $(\theta+\alpha)$-flat. Let $s$ be some rightward-pointing unit tangent to $X$ at a regular point. Then $\angle(s,E_1) \leq \theta$. Since $\phi$ preserves angles, $\angle(\phi_* s, \phi_* E_1) \leq \theta$. Thus \[\angle(\phi_* s, E_1) \leq \angle(\phi_* s, \phi_* E_1) + \angle(\phi_* E_1, E_1) \leq \theta + \alpha.\]

    Next, we will prove that $\phi(X)$ is a $k$-cover. We claim that if $s$ is a unit tangent to $X$ at a regular point which is rightward-pointing, i.e. $\langle s,E_1\rangle > 0$, then $\phi_* s$ is also rightward-pointing. To see this, by the same argument as above, we have that $\angle(\phi_*s,E_1) < \theta + \alpha < \frac{\pi}{2}$, so $\langle \phi_*s,E_1\rangle > 0$. Applying the same argument to $-s$ shows that if $\langle s,E_1\rangle < 0$ then $\langle \phi_* s, E_1\rangle < 0$.

    We claim now that for each singular point in $\phi(X)$, the total multiplicity coming in from the left is equal to the total multiplicity leaving from the right.
    Indeed, if $x \in X_\sing$ and $s_1,\ldots,s_m,s_1',\ldots,s'_{m'} \in T_x\mathbb{R}^n$ are the unit vectors pointing out of $x$ along the geodesics leaving it with $\langle s_i,E_1\rangle < 0$ and $\langle s'_{i'},E_1\rangle > 0$ for all $i,i'$, then $\langle \phi_*s_i,E_1\rangle < 0$ and $\langle \phi_*s'_{i'},E_1\rangle > 0$ for all $i,i'$. Since these are now all the unit tangents pointing out of $\phi(x)$ along the geodesics in $\phi(X)$ leaving it, and densities are preserved, this proves the claim.

    This shows that the function
    \[\sum_{x \in \phi(X)_t} \Theta(x)\]
    is constant in $t$. It must be equal to $k$; one way to see this is to observe that it is equal to the density of $\phi(X)$ at infinity, which is equal to the density of $X$ at infinity, which is $k$. Thus $\phi(X)$ is a $k$-cover.
\end{proof}

It will be convenient to use the following lemma, which has an elementary proof given in \cite{mityagin2015zerosetrealanalytic}.

\begin{lemma} \label{real-analytic-vanish}
    Suppose that $M$ is a connected real-analytic smooth manifold without boundary, and $f: M \to \mathbb{R}$ is a real-analytic function which is not identically zero. Then $\{x \in M: f(x) = 0\}$ has zero measure in $M$.
\end{lemma}
\begin{proof}
    In \cite{mityagin2015zerosetrealanalytic} this is proved for $M$ a connected open subset of $\mathbb{R}^m$. The lemma follows by considering real-analytic coordinate charts.
\end{proof}

\begin{lemma} \label{get-good}
    Fix $\alpha \in (0,\frac{\pi}{2}-\theta].$ Then there exists a $\phi \in SE_\alpha(n)$ such that $\phi(X)$ is good. This rigid motion can be taken to be arbitrarily close to the identity.
\end{lemma}

\begin{proof}
First, we want to get rid of distinct singular points of $X$ occurring at the same time, which also rules out singular ghost-crossings. Note that since $|X_\sing| < \infty$, there are finitely many pairs $x,x' \in X_\sing$, $x \neq x'$. Given such a pair $x \neq x'$, note that the map $SE(n) \to \mathbb{R}$, $\phi \mapsto \langle \phi(x)-\phi(x'),E_1\rangle$ is real-analytic and not identically zero, so by Lemma \ref{real-analytic-vanish}, \[\{\phi \in SE(n): \phi(x)_1 = \phi(x')_1\}\] has Haar measure zero. Taking the union of all of these sets over all pairs of singular points $x \neq x'$, the union is still measure zero.

Next, we aim to get rid of degenerate ghost-crossings. Since $|X_\sing| < \infty$, the set
\[\mathcal{S} = \{s \in S^{n-1}: s \text{ is tangent to } X_\reg \text{ somewhere, and } \langle s,E_1\rangle > 0 \}\]
is finite. For any pair $s \neq s'$ in $\mathcal{S}$ and $j \geq 2$, the map $SE_\alpha(n) \to \mathbb{R}$, $\phi \mapsto \frac{\phi_* s_j}{\phi_* s_1} - \frac{\phi_* s_j'}{\phi_* s_1'}$ (which is well-defined since $\phi \in SE_\alpha(n)$) is real-analytic and not identically zero, so again by Lemma \ref{real-analytic-vanish} the set
\[\left\{\phi \in SE_\alpha(n): \frac{\phi_* s_j}{\phi_* s_1} = \frac{\phi_* s_j'}{\phi_* s_1'}\right\}\]
has measure zero. The union over these sets for $s \neq s'$ in $\mathcal{S}$ still has measure zero in $SE_\alpha(n)$.

However, there may be distinct connected components of $X_\reg$ which are parallel to each other. Suppose there are two such segments $I,J$ parallel to one another, and extend them both to affine lines in $\mathbb{R}^n$. If we parameterize each segment in time as $x(t)$ and $x'(t)$ with $x_1(t) = x'_1(t) = t$, then $x(t) - x'(t)$ is a constant independent of $t$. This constant; call it $c(I,J)$, has $c_1(I,J) = 0$. If $c(I,J) \neq 0$, then for each $j \geq 2$, the map $SE_\alpha(n) \to \mathbb{R}$, $\phi \mapsto c_j(\phi(I),\phi(J))$, is real-analytic and not identically zero, except in the case where $I$ and $J$ extend to the same affine line. Thus again by Lemma \ref{real-analytic-vanish}
\[\{\phi \in SE_\alpha(n): c_j(\phi(I),\phi(J)) = 0\}\] has measure zero in $SE_\alpha(n)$.
The union over these sets for $I \neq J$ parallel has measure zero in $SE_\alpha(n)$. In the case that $c(I,J) = 0$, we have $c(\phi(I),\phi(J)) = 0$ for any $\phi$. However, since $\phi(I)$ and $\phi(J)$ are disjoint but lie on the same affine line, they can never share an $x_1$-coordinate, so there can be no ghost-crossing between them. 

Since the union over all the sets above is still measure zero in $SE_\alpha(n)$, we can apply an arbitrarily small rigid motion $\varphi \in SE_\alpha(n)$ to $X$ which avoids all of these sets. Then, each time slice $\varphi(X)_t$ contains at most one singular point (and so has no singular ghost-crossings), and $\varphi(X)$ contains no degenerate ghost-crossings. 

Finally, we get rid of half-singular ghost-crossings. If $n=2$, there are no ghost-crossings, so there is nothing to do. Thus assume $n \geq 3$. Let $SE^1(n)$ be the set of rigid motions $\phi \in SE(n)$ which restrict to rigid motions of each time slice $\pi^{-1}(t)$. These also fix $E_1$. Observe that $SE^1(n)$ is a Lie subgroup of $SE(n)$. For any $x \in \varphi(X)_\sing$, let $t = \pi(x)$ and consider the time-slice $\varphi(X)_t$. Let $\varphi(X)_t \setminus \{x\} = \{x_1,\ldots,x_m\}$, so $x_1,\ldots,x_m \in \varphi(X)_\reg.$ For each $1 \leq i \leq m$ and $j \geq 2$, the map $SE^1(n) \to \mathbb{R}$, $\phi \mapsto \langle \phi(x_i)-\phi(x),E_j\rangle$ is real-analytic and not identically zero. Thus again by Lemma \ref{real-analytic-vanish}, $\{\phi \in SE^1(n): \langle \phi(x_i) - \phi(x),E_j\rangle = 0\}$  has measure zero in $SE^1(n)$. The union of these sets over every $x \in \varphi(X)_\sing$, $i$ and $j$ still has measure zero in $SE^1(n)$. 

Thus we can apply an arbitrarily small rigid motion $\varphi' \in SE^1(n)$ to $\varphi(X)$ such that $\varphi'(\varphi(X))$ has no half-singular ghost-crossings. Since the previous bad sets we avoided were closed, by taking $\varphi'$ small enough, we do not re-enter them. The resulting ISGN $\varphi'(\varphi(X))$ has one singular point per timeslice, no degenerate ghost-crossings, and no half-singular ghost-crossings, so is good. Since we took $\varphi,\varphi'$ arbitrarily small, we can ensure that $\varphi' \circ \varphi \in SE_\alpha(n)$, and take $\phi = \varphi' \circ \varphi.$
\end{proof}

The following sums up everything we have done in this section.

\begin{proposition} \label{rotateX}
Suppose $\delta > 0$ and $\theta < \frac{\pi}{4}$, and fix $\gamma \in (0,\frac{1}{4}).$ Then there exists a rigid motion $\phi \in SE(n)$ such that $\phi(X)$ is a $k$-cover, $3\theta$-flat, $\gamma$-centered, and good.
\end{proposition}

\begin{proof}
 First apply a rigid motion $\phi$ to $X$ such that the central direction of $\phi(X)$ is $E_1$. This rigid motion can be taken to lie in $SE_\alpha(n)$ with $\alpha = \theta$. Since $\theta < \frac{\pi}{4}$, we have $\alpha \leq \frac{\pi}{2} - \theta$, so by Lemma \ref{rotation-preserves}, $\phi(X)$ is $2\theta$-flat and a $k$-cover. 
 Next, by Lemma \ref{get-good}, we may apply another rigid motion $\phi'$ to $\phi(X)$ such that $\phi'(\phi(X))$ is $\gamma$-centered and good (here we use that $\delta(\phi'(\phi(X)))$ is continuous in $\phi' \in SE(n)$ near the identity, and take $\phi'$ as close as we need to the identity). Assuming $\phi'$ is small enough, $\phi'(\phi(X))$ is $3\theta$-flat and a $k$-cover. 
\end{proof}

\section{Collision Potentials}
\label{sec:colllisionpotentials}

In this section, $X \subset \mathbb{R}^n$ is an ISGN which is a $k$-cover, $\theta$-flat, $\theta \in [0,\frac{\pi}{2})$, $\gamma$-centered, $\gamma \in [0,\frac{1}{4})$, good, and with $|X_\sing| < \infty.$

\begin{definition}
Let $\tilde V_\mathrm{ghost} \subset \mathbb{R}$ be the collection of times $t$ at which a ghost-crossing occurs (in any direction). By Lemma \ref{bound-ghost}, $\tilde V_{\mathrm{ghost}}$ is finite.
Then let $V = \tilde V \cup \tilde V_{\mathrm{ghost}}$, and let $U = \mathbb{R} \setminus V.$ Observe that $V \supset \tilde V$, $U \subset \tilde U$, and $V$ is finite.
\end{definition}

\begin{definition}
    For $j \geq 2$ and $t \in U$, given $x \neq x' \in X_t$,
    we define the $j$\emph{th collision pair potential} of $x$ and $x'$ to be 
    \[p_j(x,x') = \Theta(x)\Theta(x') \, (\langle s_{x'},E_j\rangle-\langle s_x,E_j\rangle) \sign (x_j - x'_j), \]
     where 
     \[\sign(a) = \begin{cases} 1: a > 0 \\ 
     0 : a = 0 \\ 
     -1 : a < 0
     \end{cases}\]
     is the sign function.
\end{definition}

Notice that this is invariant if we swap $x$ and $x'$.  This pair potential models how much the geodesics are pointing towards each other, along the $E_j$-axis. The idea is to have a pairwise interaction between geodesics at $x'$ and $x$ roughly consisting of the difference in their slopes in direction $j$, multiplied by the sign of the differences in their positions, so that we get a positive contribution from the pair $x,x'$ if the geodesics are pointing towards each other in direction $j$, and a negative one if they are pointing away. 

However, the contribution is not exactly the difference in slopes (the slope at $x$ in direction $j$ would be $\frac{\langle s_x, E_j\rangle}{\langle s_x,E_1\rangle}$). The current form is necessary to ensure that changes in the potential are local, i.e. coming from individual singular points and ghost crossings, but it comes at the cost that the sign of the pairwise potential is not always as expected. However, we can show that the degree to which it is the wrong sign is higher order in $\delta$, so is negligible as $\delta \to 0$.

\begin{lemma} \label{pair-potential-sign}
Assume that $\delta \leq \frac{1}{2}$.
    Suppose that $t \in U$, $x \neq x' \in X_t$, and $x'_j < x_j$. Let $s = s_x$ and $s' = s_{x'}$. If $\frac{s_j'}{s_1'} \geq \frac{s_j}{s_1}$, then 
    $p_j(x,x') \geq -8\Theta(x)\Theta(x')\delta^3$. Else if $\frac{s_j'}{s_1'} \leq \frac{s_j}{s_1}$, then $p_j(x,x') \leq 8\Theta(x)\Theta(x')\delta^3$.
    
\end{lemma}
\begin{proof}
    
        We have
        \[s'_j - s_j = s'_1\left(\frac{s'_j}{s'_1} - \frac{s_j}{s_1}\right) + (s_1'-s_1) \frac{s_j}{s_1}.\]
         We claim that the absolute value of the second term is at most $8\delta^3$. Indeed by Lemma \ref{almost-centered-slope-bounds}, $1-\delta^2 \leq s_1,s_1' \leq 1$, so $|s_1'-s_1| \leq \delta^2$, and
        \[\left|\frac{s_j}{s_1}\right| \leq \frac{2\delta}{1-\delta^2} \leq 4\delta.\] Thus our claim holds. 

        Since $x'_j < x_j$, we have $\sign(x_j-x'_j)=1.$
        If $\frac{s_j'}{s_1'} \geq \frac{s_j}{s_1}$, then
        \[p_j(x,x') = \Theta(x) \Theta(x') \, (s'_j-s_j) \geq -8 \Theta(x)\Theta(x')\delta^3. \]
        Instead if $\frac{s_j'}{s_1'} \leq \frac{s_j}{s_1}$, then
         \[p_j(x,x') = \Theta(x) \Theta(x') \, (s'_j-s_j) \leq 8 \Theta(x) \Theta(x') \delta^3. \]
\end{proof}

\begin{definition}
    For $j \geq 2$ and $t \in U$, we define the $j$\emph{th collision potential} $\Phi_j(t)$ as follows:
    \[
    \Phi_j(t) = \sum_{\{x,x'\} \subset X_t} p_j(x,x').
    \]
    This is understood as summing over unordered pairs $\{x,x'\}$,
    so each unordered pair has its potential counted once. Further, the \emph{total collision potential} is
    \[\Phi(t) = \sum_{j=2}^n \Phi_j(t).\]
   
\end{definition}

    Then we have the following global bound.

    \begin{lemma} \label{potential-global-bound}
    For all $t \in U$ and $j \geq 2$, we have 
    \[|\Phi_j(t)| \leq 2 k^2 \delta.\]
    Further,
    \[|\Phi(t)| \leq 2 nk^2\delta.\]
    \end{lemma}

    \begin{proof}
        From Lemma \ref{almost-centered-slope-bounds}, for all $x \in X_t$ and $j \geq 2$, $|\langle s_x,E_j\rangle| \leq 2\delta$. Therefore, for any $x \neq x' \in X_t$, $|p_j(x,x')| \leq 4\Theta(x)\Theta(x')\delta.$ Thus,
        \[|\Phi_j(t)| \leq  {k \choose 2} 4\delta \leq 2k^2\delta.\]
        The second bound follows by summing the first over $j \geq 2$.
    \end{proof}

    \begin{proposition} \label{potential-local} Fix $j \geq 2$. Then
        $\Phi_j$ is constant on each connected component of $U$. Further, at $t \in V$, consider the jump $\Delta_j(t) = \Phi_j(t_1) - \Phi_j(t_0)$ where $t_0,t_1 \in U$, $t_0 < t < t_1$ and $[t_0,t_1] \cap V = \{t\}$. Since $X$ is good, there is at most one singular point in $X_t$.
        Let $\mathcal{G}_j(t)$ be the set of ghost-crossings in direction $j$ at time $t$. Then the jump $\Delta_j(t)$ can be decomposed as a sum
        \[\Delta_j(t) = \Delta_j^{\sing}(t) + \sum_{(\{x,x'\},j) \in \mathcal{G}_j(t)} \Delta_j^{\mathrm{ghost}}(x,x',t),\]
        where $\Delta_j^\sing(t)$ is the difference of pair-potentials only from points on geodesics which hit the singular point at $t$ (if it exists), and $\Delta_j^{\mathrm{ghost}}(x,x',t)$ is change in pair potential $p_j(x,x')$ across the time $t$ (where we move the points $x,x'$ slightly forward and backward in time along the geodesics). Note that we only count each equivalence class of ghost-crossing once.
    \end{proposition}

    \begin{proof}
    For $i \in \{0,1\}$ let $Z^i_{\sing}$ be either the set of points in $X_{t_i}$ whose associated geodesics hit the unique singular point in $X_t$, if it exists, or the empty set, if there is no singular point in $X_t$. Let $Z^i_\reg = X_{t_i} \setminus Z^{i}_\sing.$ Because the geodesics do not hit any singular points, there is a bijection $\phi: Z^{0}_\reg \to Z^1_\reg$, so when the index $i$ does not matter, we call it $Z_\reg.$

    Because $X$ is good, there are no half-singular ghost-crossings. That means that if there is a singular point $x_\sing \in X_t$, then for all $x \in X_t \setminus \{x_\sing\}$, $x_j \neq (x_{\sing})_j$, so we do not have to worry about the indicator function changing for pair potentials which are between $Z^i_\reg$ and $Z^i_\sing.$ 
    
    Suppose first that there is a singular point in $X_t$.
    To show that pair potentials between geodesics which hit the singular point and those which don't 
    yield no change to $\Delta_j(t)$, 
    fix one $x^0_\reg \in Z^{0}_\reg$ and let $x^1_\reg = \phi(x^0_\reg).$ Then, up to a sign which doesn't change,

    \[\sum_{x^0_\sing \in Z^0_\sing} p_j(x^0_\sing, x^0_\reg) - \sum_{x^1_\sing \in Z^1_\sing} p_j(x^1_\sing, x^1_\reg)\]
    \[
= 
   \left( \pm \sum_{x^0_\sing \in Z^0_\sing} \Theta(x^0_\sing) \Theta(x^0_\reg) \, (\langle s_{x^0_\sing}, E_j\rangle - \langle s_{x^0_\reg}, E_j \rangle) \right) \] \[ - \left( \pm \sum_{x^1_\sing \in Z^1_\sing} \Theta(x^1_\sing) \Theta(x^1_\reg) \, (\langle s_{x^1_\sing}, E_j\rangle - \langle s_{x^1_\reg}, E_j \rangle) \right)
    \]

\[
= 
   \left( \pm \Theta(x^0_\reg) \sum_{x^0_\sing \in Z^0_\sing} \Theta(x^0_\sing)  \, \langle s_{x^0_\sing}, E_j\rangle \right)  - \left( \pm \Theta(x^1_\reg) \sum_{x^1_\sing \in Z^1_\sing} \Theta(x^1_\sing)  \, \langle s_{x^1_\sing}, E_j\rangle \right)
    \]
    \[
-  \left( \pm \Theta(x^0_\reg) \langle s_{x^0_\reg},E_j\rangle  \sum_{x^0_\sing \in Z^0_\sing} \Theta(x^0_\sing) \right)  + \left( \pm \Theta(x^1_\reg)\langle s_{x^1_\reg},E_j\rangle \sum_{x^1_\sing \in Z^1_\sing} \Theta(x^1_\sing)  \right)
    \]

    But this is zero, because $\Theta(x^0_\reg) = \Theta(x^1_\reg)$, $s_{x^0_\reg} = s_{x^1_\reg}$, 
     the stationary condition at the singular point, and since the multiplicity entering the singular point is equal to the multiplicity leaving (by the $k$-cover condition).

    Since the slopes of all geodesics not hitting the singular point in $X_t$ do not change past $t$, the only pair-potentials which change are from changes in the sign function $\sign(x'_j-x_j)$, which are from ghost-crossings. That completes the proof.
    \end{proof}

    The following lemma tells us that the change of pair potential at a ghost-crossing cannot be large and positive.

    \begin{lemma} \label{ghost-crossing-potential}
    Assume that $\delta \leq \frac{1}{2}$.
        Suppose that $t \in V$ and that $(x,x',j)$ is a ghost-crossing at time $t$. Then in the notation of Proposition \ref{potential-local}, we have
        \[\Delta_j^{\mathrm{ghost}}(x,x',t) \leq 16 \Theta(x) \Theta(x') \delta^3.\]
    \end{lemma}

    \begin{proof}
        Like before, choose $t_0 < t < t_1$ with $[t_0,t_1] \cap V = \{t\}.$ Let $s = s_x$, $s' = s_{x'}$. Let $x_0,x'_0 \in X_{t_0}$ be the points on the geodesics which will reach $x,x'$, and similarly $x_1,x'_1 \in X_{t_1}$. Without loss of generality, we may assume that $(x_0')_j < (x_0)_j$. Then since the geodesics will cross in the $j$th coordinate,
        we must have $\frac{s'_j}{s'_1}> \frac{s_j}{s_1}$. 
        By Lemma \ref{pair-potential-sign}, it follows that $p_j(x_0,x'_0) \geq -8\Theta(x)\Theta(x')\delta^3$, and $p_j(x_1,x'_1) \leq 8\Theta(x)\Theta(x')\delta^3$. Since
        $\Delta_j^{\mathrm{ghost}}(x,x',t) = p_j(x_1,x'_1)-p_j(x_0,x'_0),$
        we are done.
    \end{proof}

    Next, we show that at a singular point, the collision potential must decrease proportionally to a scale quantity of the singular point, which we call its strength, as long as the strength is large at scale $\delta$.

    \begin{definition}
        The \emph{strength} of a singular point $x \in X_\sing$, denoted by $\strength(x)$, is defined as follows. Let $s_1,\ldots,s_\ell$ be the rightward-pointing unit tangent vectors to $X$ which point into $x$ from the left (with respect to the $E_1$-direction). Then
        \[\strength(x) = \max_{i \neq j}|s_i-s_j|.\]
    \end{definition}

    \begin{proposition} \label{strength-pays-potential} Assume that $\delta \leq \frac{1}{2}.$
    Suppose that $t \in V$ and there is a singular point $q \in X_t$. 
    For all $j \geq 2$, we have
\[\Delta^\sing_{j}(t) \leq 16 {\Theta(q) \choose 2}\delta^3.\]
Further, writing $\beta = \strength(q)$,
 suppose that $\beta > \delta^2$ and
    $\sqrt{\frac{\beta^2 - \delta^4}{n}} > 8\delta^3.$
    Then there is some $j^* \geq 2$ such that
    \[\Delta^\sing_{j^*}(t) \leq 16 {\Theta(q) \choose 2}\delta^3- \sqrt{\frac{\beta^2-\delta^4}{n}}.\]
    \end{proposition}

    \begin{proof}
        As before fix $t_0 < t < t_1$ with $[t_0,t_1] \cap V =\{t\}$. For $i \in \{0,1\}$ let $Z^i$ be the set of points in $X_{t_i}$ whose associated geodesics hit the singular point in $X_t.$ We claim that for all $j \geq 2$ and $x \neq x' \in Z^0$, $p_j(x,x') \geq -8 \Theta(x)\Theta(x') \delta^3.$ Indeed, this follows directly from Lemma \ref{pair-potential-sign}, since the geodesics meeting at the singular point yields the right slope comparison, as they are coming together. Similarly, for all $j \geq 2$ and $x \neq x' \in Z^1$, $p_j(x,x') \leq 8 \Theta(x) \Theta(x') \delta^3.$ Summing, we find that for any $j \geq 2$,
        \[\Delta_j^{\sing}(t) = \left(\sum_{\{x,x'\} \subset Z^1} p_j(x,x')\right) - \left(\sum_{\{x,x'\} \subset Z^0} p_j(x,x') \right) \leq 2 {\Theta(q) \choose 2} 8 \delta^3 = 16 {\Theta(q) \choose 2} \delta^3.\]

        Now, suppose that $\beta = \strength(q)$ satisfies $\beta > \delta^2$ and 
        $\sqrt{\frac{\beta^2 - \delta^4}{n}} > 8\delta^3.$
        Choose $x,x' \in Z^0$ with $|s_x-s_{x'}| = \beta$. Then \[\beta^2 = \sum_{j=1}^n \langle s_x-s_{x'},E_j\rangle^2.\]
        By Lemma \ref{almost-centered-slope-bounds}, 
        \[1 -\delta^2 \leq \langle s_x,E_1\rangle, \langle s_{x'},E_1\rangle \leq 1,\]
        so $|\langle s_x-s_{x'},E_1\rangle| \leq \delta^2.$ Therefore 
\[\sum_{j=2}^n \langle s_x-s_{x'},E_j\rangle^2 \geq \beta^2 - \delta^4.\]
By properties of sums, there must be some $j^* \geq 2$ such that
\[\langle s_x-s_{x'}, E_{j^*}\rangle^2 \geq \frac{\beta^2-\delta^4}{n-1} \geq \frac{\beta^2-\delta^4}{n},\]
so for $j^*$,
\[|\langle s_x - s_{x'},E_{j^*}\rangle| \geq \sqrt{\frac{\beta^2-\delta^4}{n}}.\]
Since $t_0 \in U$, there is no ghost-crossing at time $t_0$, so $x_{j^*} \neq x_{j^*}'$. Relabel so that $x'_{j^*} < x_{j^*}$. Since the $j^*$th coordinate of the geodesics will be equal at time $t$, we have
$\frac{\langle s_{x'},E_{j^*}\rangle}{\langle s_{x'},E_1\rangle} > \frac{\langle s_{x},E_{j^*}\rangle}{\langle s_{x},E_1\rangle}$.
Then by Lemma \ref{pair-potential-sign}, 
$\langle s_{x'}-s_{x},E_{j^*}\rangle \geq -8 \delta^3.$
Since we assumed that  $\sqrt{\frac{\beta^2 - \delta^4}{n}} > 8\delta^3$, we can remove the absolute value:
\[\langle s_{x'} - s_{x},E_{j^*}\rangle \geq \sqrt{\frac{\beta^2-\delta^4}{n}}.\]
All of the other terms are bounded in the same way as before. Summing we obtain that
 \[\Delta^\sing_{j^*}(t) \leq 16 {\Theta(q) \choose 2}\delta^3- \sqrt{\frac{\beta^2-\delta^4}{n}}.\]
    \end{proof}

    We can combine the above results as follows.

    \begin{proposition} \label{potential-jump-final}
        Assume that $\delta \leq \frac{1}{2}$. Suppose that $t \in V$. Then, for any $j \geq 2$, 
        \[\Delta_j(t) \leq 16{k \choose 2}\delta^3 \leq 8k^2\delta^3.\]
        Therefore, the jump $\Delta(t) = \sum_{j=2}^n \Delta_j(t)$ of $\Phi$ satisfies
        \[\Delta(t) \leq 8nk^2\delta^3.\] If further $t \in \tilde V$, and the singular point $x \in X_t$ has $\strength(x) = \beta$ satisfying $\beta > \delta^2$ and $\sqrt{\frac{\beta^2-\delta^4}{n}} > 8\delta^3$, then 
        \[\Delta(t) \leq 8nk^2\delta^3 - \sqrt{\frac{\beta^2-\delta^4}{n}}.\]
    \end{proposition}
        \begin{proof}
            First suppose that $X_t$ contains no singular points of $X$. Then $\sum_{x \in X_t}\Theta(x) = k$, so $\sum_{\{x,x'\} \subset X_t} \Theta(x)\Theta(x') \leq {k \choose 2}. $ Along with Lemma \ref{ghost-crossing-potential}, this shows that for each $j \geq 2$, $\Delta_j(t) \leq 16 {k \choose 2} \delta^3 \leq 8k^2\delta^3$. Now suppose that $X_t$ has one singular point of $X$, along with Proposition \ref{strength-pays-potential}, we can treat the singular point similarly, and we get the same conclusion. Finally, assuming that $\strength(x) = \beta$ and $\beta > \delta^2$ and $\sqrt{\frac{\beta^2-\delta^4}{n}} > 8\delta^3$, the same argument holds again, except this time we have the extra term $-\sqrt{\frac{\beta^2-\delta^4}{n}},$ which completes the proof.
        \end{proof}

\section{Spreads and Merging}
\label{sec:spreadsmerging}

In this section, $X \subset \mathbb{R}^n$ is an ISGN which is a $k$-cover, $\theta$-flat, $\theta \in [0,\frac{\pi}{2})$, $\gamma$-centered, $\gamma \in [0,\frac{1}{4})$, good, and with $|X_\sing| < \infty.$ All unit tangent vectors to $X$ will always be taken to be rightward pointing, with respect to the $E_1$ direction.

\begin{definition}
    Fix $t \in \tilde U$ and an arbitrary subset $Z \subset X_t$. We define the \emph{spread} $\spread(Z)$ of $Z$ by 
    \[\spread(Z)^2 = \frac{1}{2} \sum_{x,x' \in Z} \Theta(x)\Theta(x') |s_x-s_{x'}|^2. \]
\end{definition}

The spread is conserved past singular times, as long as we don't bring in new geodesics.

\begin{proposition}
    Suppose that $I \subset \mathbb{R}$ is connected. Let $Z$ be a union of connected components of $X_I.$ For any $t \in I$, let $Z_t = Z \cap \pi^{-1}(t)$. Then for any $t,t' \in I \cap \tilde U$, $\spread(Z_t) = \spread(Z_{t'})$. Thus the spread is an invariant quantity of $Z$, which we also call its spread, $\spread(Z).$
\end{proposition}

\begin{proof}
    Consider $Z$ as the $I$-time slab of a new (not necessarily nearly centered) $\theta$-flat ISGN in $\mathbb{R}^n$, by extending the geodesics on either end of $I$ out to infinity. Since at every singular point of $Z$, the total multiplicity of geodesics leaving on either side, with respect to the $E_1$-direction, is equal, $Z$ is an $\ell$-cover for some $\ell \leq k$. Thus we can apply our theory, with $X$ as this extended $Z$. Conservation then follows by Lemma \ref{spread-formula}. 
\end{proof}

\begin{definition}
    Suppose that $I \subset \mathbb{R}$ is connected. As above, any union $Z$ of connected components of $X_I$ is an $\ell$-cover for some $\ell \leq k$. Let the \emph{total multiplicity} of $Z$, denoted by $\Theta(Z)$, be equal to $\ell$. Notice that $\Theta(X_I) = k.$
\end{definition}

Next, we will define a series of partitions, which track a merging process forward in time.

\begin{definition}
    Fix $t_0 \in \tilde U$. For each $t \geq t_0$, let $P_t$ be the set of connected components of $X_{[t_0,t]}$. $P_t$ induces a partition $p_t$ of $X_t$, by taking the intersection of each component with $X_t$.
\end{definition}

Observe that $p_{t_0} = \{\{x\}: x \in X_{t_0}\}.$ 

\begin{definition}
    We say that $X$ \emph{merges completely} after $t_0$ if $X_{[t_0,\infty)}$ is connected. Otherwise, we say that $X$ \emph{splits} after $t_0.$
\end{definition}

So far we have not excluded the case where $X_{t_0}$ is connected. That is, $X_{t_0}$ is a singleton. But since $t_0 \in \tilde U$, then $\delta = 0$, and then $X$ is just a multiplicity $k$ line. We will continue not to technically exclude this case.

\begin{proposition} \label{merge-spread-estimate}
    Suppose that $t \in \tilde V$, $t > t_0$ is a singular time, and let $x_\sing \in X_t$ be the unique singular point in that timeslice. Let $Z$ be a connected component of $X_{[t_0,t]}$, and let $z_1,\ldots,z_m$ be the connected components of $Z \cap X_{[t_0,t)}$. If $m=1$, then $\spread(Z) = \spread(z_1)$. Else if $m \geq 2$, we have a ``merging'' spread estimate
    \[\spread(Z)^2 \leq \left( \sum_{i=1}^m \spread(z_i)^2 \right) + 2\left( \sum_{1 \leq i < i' \leq m} \Theta(z_i)\Theta(z_{i'})(\spread(z_i)+\spread(z_{i'}))^2 \right) + \bigg(\Theta(Z)^2 \strength(x_\sing)^2 \bigg).\]
\end{proposition}

\begin{proof} Let $x_\sing$ be the singular point of $X$ in $X_t$.
    Choose $t' < t$ such that $[t',t] \cap \tilde V = \{t\}.$ Then
    \[\spread(Z)^2 = \frac{1}{2}\sum_{x,y\in Z_{t'}} \Theta(x) \Theta(y) |s_x-s_y|^2.\]
    We can rewrite this as 
    \[\spread(Z)^2 = \frac{1}{2} \left[\sum_{i=1}^m \sum_{x,y \in z_{i,t'}} \Theta(x) \Theta(y) |s_x-s_y|^2 \right] + \left[ \sum_{1 \leq i < i' \leq m} \sum_{\substack{x \in z_{t',i} \\ y \in z_{t',i'}}} \Theta(x) \Theta(y) |s_x-s_y|^2\right].\]
    Notice that in the left term,

    \[\frac{1}{2} \sum_{x,y \in z_{i,t'}} \Theta(x)\Theta(y) |s_x-s_y|^2 = \spread(z_i)^2.\]

    To bound the right term, fix $i < i'$ and $x \in z_{t',i}$, $y \in z_{t',i'}$. Let $x_\hit \in z_{i,t'}$ be on a geodesic that will hit $x_\sing$, and $y_{\hit} \in z_{t',i'}$ be on a (necessarily distinct) geodesic that will hit $x_\sing$. We can guarantee this because $x_\sing$ is the only singular point in $X_t$, and $Z$ is connected. By the definition of the spread of $z_i$ and $z_{i'}$ and throwing away all the other terms, we have
    $|s_x-s_{x_\hit}| \leq \spread(z_i)$ and $|s_y-s_{y_\hit}| \leq \spread(z_{i'})$. By the definition of the strength of the collision at $x_\sing$, we have $|s_{x_\hit} - s_{y_{\hit}}| \leq \strength(x_\sing).$
    By the triangle inequality, it follows that
    \[|s_x-s_y| \leq \spread(z_i) + \spread(z_{i'}) + \strength(x_\sing).\]
    By the elementary inequality $(a+b)^2 \leq 2a^2+2b^2$ for $a,b \geq 0$, we obtain 
     \[|s_x-s_y|^2 \leq 2(\spread(z_i) + \spread(z_{i'}))^2 + 2\strength(x_\sing)^2.\]
     Summing the bound over all $x \in z_{t',i} $ and $y \in z_{t',i'}$ with multiplicity, we have
     \[\sum_{\substack{x \in z_{t',i} \\ y \in z_{t',i'}}} \Theta(x) \Theta(y) |s_x-s_y|^2 \leq \Theta(z_i)\Theta(z_{i'})\big[2(\spread(z_i) + \spread(z_{i'}))^2 + 2 \strength(x_\sing)^2\big].\]
     Summing over $i,i'$ we have
     \[\sum_{1 \leq i < i' \leq m} \sum_{\substack{x \in z_{t',i} \\ y \in z_{t',i'}}} \Theta(x) \Theta(y) |s_x-s_y|^2 \] \[ \leq 2 \sum_{1 \leq i < i' \leq m} \Theta(z_i) \Theta(z_{i'}) (\spread(z_i)+\spread(z_{i'}))^2 + 2 \sum_{1 \leq i < i' \leq m} \Theta(z_i)\Theta(z_{i'}) \strength(x_\sing)^2.\]
     But 
     \[2\sum_{1 \leq i < i' \leq m} \Theta(z_i)\Theta(z_{i'}) \leq  \sum_{i,i'} \Theta(z_i)\Theta(z_{i'}) =  \left( \sum_{i=1}^m \Theta(z_i)\right) \left( \sum_{i'=1}^m \Theta(z_{i'})\right) = \Theta(Z)^2.\]
     Thus we have the bound
     \[\sum_{1 \leq i < i' \leq m} \sum_{\substack{x \in z_{t',i} \\ y \in z_{t',i'}}} \Theta(x) \Theta(y) |s_x-s_y|^2 \leq \Theta(Z)^2 \strength(x_\sing)^2 + 2 \sum_{1 \leq i < i' \leq m} \Theta(z_i) \Theta(z_{i'}) (\spread(z_i)+\spread(z_{i'}))^2.\]
     Putting everything together, the proof of the estimate is complete.
\end{proof}

The idea is to show by induction on Proposition \ref{merge-spread-estimate} that the spread of any connected component is bounded by the maximum strength of singular points along it.

\begin{proposition} \label{merge-spread-induction}
    Suppose that $t \in \tilde U$, $t \geq t_0$, and let $Z$ be a connected component of $X_{[t_0,t]}$. Let $F: \mathbb{N}_+ \to \mathbb{R}$ be the unique function satisfying
    \[F(1) = 0, \;\;\; F(\ell) = \sqrt{\ell^2 + \ell(1 + 4\ell^3)F(\ell-1)^2}, \; \ell \geq 2.\] Observe that $F$ is nondecreasing.
    Then 
    \[\spread(Z) \leq F(\Theta(Z)) \, \max_{x \in Z \cap X_\sing} \strength(x),\]
    where we take the maximum to be zero if $Z \cap X_\sing = \varnothing$.
\end{proposition}

\begin{proof}
    Let $\beta = \max_{x \in Z \cap X_\sing} \strength(x)$. We induct on the total multiplicity $\Theta(Z)$. If $\Theta(Z) = 1$, then $Z$ is a single geodesic segment, so $\spread(Z) = 0$. Assume that $\ell = \Theta(Z) \geq 2$ and that $Z_{t_0}$ is disconnected; otherwise $\spread(Z) = 0$ again. Let $t_\merge$ be the supremum of times $t'$ such that $Z_{[t_0,t']}$ is disconnected. Then $t_0 < t_\merge < t$ and $t_\merge \in \tilde V$.
    Let $z_1,\ldots,z_m$ be the connected components of $Z_{[t_0,t_\merge)}$, so $m \geq 2$. Let $x_\sing$ be the unique singular point in $X_{t_\merge}$. Then $\Theta(z_i) < \ell$ for all $i$, so they satisfy the inductive hypothesis when restricted to the interval $[t_0,t_{\mathrm{new}}]$, where $t_{\mathrm{new}} < t_\merge$ and $[t_{\mathrm{new}},t_\merge] \cap \tilde V = \{t_\merge\}.$ Applying Proposition \ref{merge-spread-estimate} we have 
    \[\spread(Z)^2 \leq \left( \sum_{i=1}^m \spread(z_i)^2 \right) + 2\left( \sum_{1 \leq i < i' \leq m} \Theta(z_i)\Theta(z_{i'})(\spread(z_i)+\spread(z_{i'}))^2 \right) + \bigg(\Theta(Z)^2 \strength(x_\sing)^2 \bigg).\]
    Simplify by bounding $\Theta(z_i) \leq \ell$, and $\strength(x_\sing) \leq \beta$:
    \[\spread(Z)^2 \leq \left( \sum_{i=1}^m \spread(z_i)^2 \right) + 2\ell^2\left( \sum_{1 \leq i < i' \leq m} (\spread(z_i)+\spread(z_{i'}))^2 \right) + \ell^2 \beta^2.\]
    By the inequality $(a+b)^2 \leq 2a^2 + 2b^2$ for $a,b \geq 0$,
    \[ \sum_{1 \leq i < i' \leq m} (\spread(z_i)+\spread(z_{i'}))^2 \leq 2\sum_{1 \leq i < i' \leq m} \left(\spread(z_i)^2+\spread(z_{i'})^2\right) \leq \sum_{i,i'} \left(\spread(z_i)^2+\spread(z_{i'})^2\right)\]
    \[= 2m\sum_{i=1}^m \spread(z_i)^2.\]
    Therefore we can further bound
    \[\spread(Z)^2 \leq \ell^2\beta^2 + (1 + 4m\ell^2) \sum_{i=1}^m \spread(z_i)^2.\]
    Since $m \leq \ell$ and $\spread(z_i) \leq F(\ell-1) \beta$ for all $i$, we obtain 
    \[\spread(Z)^2 \leq \ell^2\beta^2 + \ell(1 + 4\ell^3) F(\ell-1)^2\beta^2, \]
    so
     \[\spread(Z) \leq \sqrt{\ell^2 + \ell(1 + 4\ell^3) F(\ell-1)^2} \beta. \]
     Thus we can take $F(\ell) = \sqrt{\ell^2 + \ell(1+4\ell^3 )F(\ell-1)^2},$
     which is what we assumed.
\end{proof}

Putting the bound together with Lemma \ref{spread-formula}, we obtain the following result. This shows us that in order for $X$ to merge completely after some $t_0 \in \tilde U$, there must be a strong collision in the time $X$ merges back together from $t_0$.

\begin{proposition} \label{find-strong-singular-point}
    Assume that $k \geq 2$ and $\delta \leq 1$. Suppose that $X$ merges completely after $t_0$. Let $t_* \in [t_0, \infty)$ be the infimum of $t \geq t_0$ such that $X_{[t_0,t]}$ is connected. Then
    \[\max_{x \in X_\sing \cap X_{[t_0,t_*]}} \strength(x) \geq \frac{\sqrt{\frac{k}{2}}}{F(k)} \delta,\]
    where we take the maximum to be zero if $X_\sing \cap X_{[t_0,t_*]} = \varnothing.$
\end{proposition}
\begin{proof}
    By Lemma \ref{spread-formula}, we have
    \[\spread(X_{[t_0,t_*]}) = \sqrt{k\delta^2 - \frac{\delta^4}{4}}.\]
    But by Proposition \ref{merge-spread-induction}, 
    \[\spread(X_{[t_0,t_*]}) \leq F(k) \max_{x \in X_\sing \cap X_{[t_0,t_*]}} \strength(x).\]
    Since $k \geq 2$, $F(k) > 0$, so
    \[ \max_{x \in X_\sing \cap X_{[t_0,t_*]}} \strength(x) \geq \frac{\sqrt{k\delta^2 - \frac{\delta^4}{4}}}{F(k)}.\]
    Since $\delta \leq 1$, we obtain the desired result.
\end{proof}

\section{The Global Argument}
\label{sec:globalargument}

In this section, we give a proof of Theorem \ref{maintechnicaltheorem}. First, we will give a proof sketch. 

\begin{proof}[Proof sketch] The proof is by induction on $k \geq 1$. If $k=1$ or $\delta = 0$, or $X_\sing = \varnothing$, the proof is trivial, so we assume $k \geq 2$ and $\delta > 0$, and $X_\sing \neq \varnothing$. Thus $V \neq \varnothing$. By Proposition \ref{potential-jump-final}, if we ignore third-order terms in $\delta$, the collision potential $\Phi$ is nonincreasing, and decreases by an amount linear in the strength $\beta$ when there is a singular with strength large compared to  $\delta^2$ ($\beta$ being linear in $\delta$ works). By Lemma \ref{potential-global-bound}, the collision potential is bounded itself linearly in $\delta$. Lastly, Proposition \ref{find-strong-singular-point} gives a method of finding ``strong'' singular points whose strength is at least linear in $\delta$, as long as $X$ merges completely after some given start time $t_0$. 

If we could ignore the contributions from weaker singular points to $\Phi$ that may be the wrong sign, combining the global bound on $\Phi$ with the decrease in $\Phi$ from each strong singular point, which are both linear in $\delta$, we have a bound on the number of strong singular points. The problem is that there may be many more weaker singular points and ghost crossings, which contribute the wrong sign to changes in $\Phi$. However, we can get around this by using the induction hypothesis on disconnected time-slabs of $X$ to show that the number of weak singular points and ghost-crossings does not exceed a constant times the number of strong singular points. Thus, we obtain a bound on the number of strong singular points, which therefore bounds the total number of singular points.

\end{proof}

\begin{proof}
    The proof is by induction on $k \geq 1$. We are given an ISGN $X \subset \mathbb{R}^n$ which is a $k$-cover, $\theta$-flat, and has $|X_\sing| < \infty$. 
    
    If $k=1$, $X$ is an affine line, so $|X_\sing| = 0$, so we may take $\theta(n,1)$ arbitrarily in $(0,\frac{\pi}{2})$, and $C(n,1) = 0$.
    
    Assume $k \geq 2$ and the theorem holds for all $\ell < k$. Our aim is to select $\theta = \theta(n,k)$ and $C = C(n,k)$ such that the desired conclusion holds. If $\delta = 0$, then $X$ is a union of $k$ parallel lines, counted with multiplicity, so $|X_\sing| = 0$. Thus we may assume that $\delta > 0$. If $X_\sing = \varnothing$, there is nothing to prove, so assume that $X_\sing \neq \varnothing$. Thus $V \neq \varnothing$.
    
    We may assume that $\theta \leq \frac{\theta(n,k-1)}{3},$ and will always take $C(n,\ell)$ nondecreasing in $\ell$. By induction it follows that 
    $\theta \leq \frac{1}{3}\min_{1 \leq \ell < k}\theta(n,\ell).$

    Decreasing $\theta$ if necessary, assume that $\theta < \frac{\pi}{4}.$ Fix some $\gamma \in (0,\frac{1}{4})$. By Proposition \ref{rotateX}, there exists a rigid motion $\phi \in SE(n)$ such that $\phi(X)$ is an ISGN in $\mathbb{R}^n$ which is a $k$-cover, $3\theta$-flat, $\gamma$-centered, and good. Rename $\phi(X)$ to be $X$.

    Decreasing $\theta$ if necessary such that $\theta \leq \frac{1}{12\sqrt{k}}$, by Lemma \ref{central-tilt-equivalence}, we have $\delta \leq \frac{1}{2}$.

    By Lemma \ref{potential-global-bound}, we have $|\Phi(t)| \leq 2nk^2\delta$ for all $t \in U$. Since $\Phi$ is constant outside a sufficiently large compact set, define $\Phi(\infty) = \lim_{t \to +\infty} \Phi(t)$ and $\Phi(-\infty) = \lim_{t \to -\infty} \Phi(t)$. Then \[\Phi(-\infty) - \Phi(\infty) \leq |\Phi(-\infty)| + |\Phi(\infty)| = 4nk^2\delta.\]

    For some $N \geq 0$ define a finite sequence 
    \[a_0 < b_0 \leq c_0 < a_1 < b_1 \leq c_1 < \cdots < c_{N-1} < a_N,\]
    where $a_i \in U$ and $b_i,c_i \in \tilde V$, as follows.
    Let $a_0 < \min V$, so $X_{(-\infty,a_0]}$ has no singular points nor ghost-crossings. Suppose that $a_0 < b_0 \leq c_0 < \cdots < c_{i-1} < a_i$ have been defined. If $X$ splits after $a_i$, then let $N = i$ and terminate. Else, $X$ merges completely after $a_i$. Let $c_i \geq a_i$ be the infimum of $t \geq a_i$ such that $X_{[a_i,t]}$ is connected. Then $c_i < \infty$ because $X$ merges completely after $a_i$, and $c_i > a_i$ because otherwise $\delta = 0$, contrary to our assumption. Further, $c_i \in \tilde V$. By Proposition \ref{find-strong-singular-point}, we have 
    \[\max_{x \in X_{\sing} \cap X_{[a_i,c_i]}} \strength(x) \geq \frac{\sqrt{\frac{k}{2}}}{F(k)} \delta.\]
    Let $x_i \in X_{\sing} \cap X_{[a_i,c_i]}$ achieve the maximum, and let $b_i = \pi(x_i)$. Then $b_i \in (a_i,c_i]$. Choose $a_{i+1} > c_i$ such that $(c_i,a_{i+1}] \subset U$, and recurse.
    This process must terminate, as $\tilde V$ is finite.

    By construction, for each $0 \leq i < N$, $X_{[a_i,c_i)}$ is not connected. Given a connected component $Z$ of $X_{[a_i,c_i)}$, we may extend $Z$ to an ISGN $Z'$ in $\mathbb{R}^n$ which is $3\theta$-flat and an $\ell$-cover, where $\ell = \Theta(Z) \leq k-1$, and which has finitely many singular points.
    Thus $Z'$ is $\theta(n,\ell)$-flat, an $\ell$-cover, and has $|Z_\sing| < \infty$, so we may apply the inductive hypothesis to conclude that $|Z_\sing| \leq |Z'_\sing| \leq C(n,\ell) \leq C(n,k-1).$ Since $X_{[a_i,c_i)}$ has at most $k$ connected components, we conclude that $X_{[a_i,c_i)}$ has at most $k C(n,k-1)$ singular points. There is one more singular point in $X_{c_i}$, so $X_{[a_i,c_i]}$ has at most $k C(n,k-1) + 1$ singular points.
    By the same argument, as $X_{[a_N,\infty)}$ is not connected, $X_{[a_N,\infty)}$ also has at most $k C(n,k-1)$ singular points.

    Thus, $X$ has at most $(N + 1) (k \, C(n,k-1) + 1)$ singular points. Since $X$ is good, by Lemma \ref{bound-ghost}, $X$ has at most 
    \[((N+ 1) (k \, C(n,k-1) + 1) + 1)(n-1){k \choose 2}\]
    ghost-crossings.
    Let
    \[A(n,k) = (k \, C(n,k-1) + 1)(n-1){k \choose 2} + k \, C(n,k-1) + 1,  \]
    so $X$ has at most $(N + 1) A(n,k) + (n-1){k \choose 2}$ ghost-crossings and singular points.

    Looking at the times $b_0 < \cdots < b_{N-1}$, $X$ has at least $N$ singular points which have strength at least $\beta := \frac{\sqrt{\frac{k}{2}}}{F(k)} \delta.$ 

    Since we are still free to choose $\theta$ and $\delta \leq 6\sqrt{k} \theta$ by Lemma \ref{central-tilt-equivalence}, make $\theta$ sufficiently small so that $\beta > \delta^2$ and $\sqrt{\frac{\beta^2-\delta^4}{n}}>8\delta^3$.
    
    By Proposition \ref{potential-jump-final}, we have an upper bound on the jump of $\Phi$ at every singular or ghost-crossing time in $V$. Putting the above two bounds together, we obtain that

    \[\Phi(\infty) - \Phi(-\infty) = \sum_{t \in V} \Delta(t) \leq 
    \left((N + 1) A(n,k) + (n-1){k \choose 2}\right) \, 8nk^2\delta^3 - N\sqrt{\frac{\beta^2-\delta^4}{n}}. \]

    But from above, we know that $\Phi(-\infty) - \Phi(\infty) \leq 4nk^2\delta$. Putting things together, it follows that

    \[-4nk^2\delta \leq \Phi(\infty) - \Phi(-\infty) \leq 
    \left((N + 1) A(n,k) + (n-1){k \choose 2}\right) \, 8nk^2\delta^3 - N\sqrt{\frac{\beta^2-\delta^4}{n}}, \]
    and reversing signs, we get that
    \[ N\sqrt{\frac{\beta^2-\delta^4}{n}} - \left((N + 1) A(n,k) + (n-1){k \choose 2}\right) \, 8nk^2\delta^3 \leq 4nk^2\delta.\]
    Dividing through by $\delta$ and rearranging, we have 
\[ N\left[\sqrt{\frac{\frac{k}{2F(k)^2}-\delta^2}{n}} -  A(n,k) \, 8nk^2\delta^2\right] - \left(A(n,k) + (n-1){k \choose 2}\right) \, 8nk^2\delta^2 \leq 4nk^2.\]

We claim now that \[N \leq 16\sqrt{2}(nk)^{3/2} F(k).\] Indeed, if $N = 0$ there is nothing to prove, so assume $N \geq 1$. We are still free to choose $\theta$, and $\delta \leq 6\sqrt{k} \theta$ by Lemma \ref{central-tilt-equivalence}, so take $\theta$ small enough so that in each of the three subtractions above with terms involving a $\delta^2$, at most one half of the term is lost (this can be done purely depending on $n$ and $k$, since $N \geq 1$). With $\theta$ chosen this way, it follows that 
    \[ \frac{N}{2} \sqrt{\frac{\frac{k}{4F(k)^2}}{2n}} \leq  4nk^2,\]
    so
    \[N \leq 8nk^2 \sqrt{\frac{8nF(k)^2}{k}} = 16\sqrt{2}(nk)^{3/2} F(k) \]
    as desired.
    Thus, $X$ has at most
    \[(N+1) (k \, C(n,k-1) + 1) \leq \left[ 16\sqrt{2}(nk)^{3/2} F(k) + 1\right] (k \, C(n,k-1) + 1) =: C(n,k) \]
    singular points.
\end{proof}

    \section{Proof of Main Theorem}
    \label{sec:proofmaintheorem}

    We elect to prove the following theorem, which implies Theorem \ref{maintheoremfinal}.

    \begin{theorem}
        For any $n \geq 2$ and $k \geq 1$, there exists a $C = C(n,k)$ with the following property.
        Let $(X_i)$ be a sequence of integral stationary geodesic networks in $\mathbb{R}^n$ which converge weakly (as Radon measures) to a multiplicity $k$ line $\ell$. Then for any finite length, finite radius tube $T$ centered around $\ell$ and constants $\theta,r > 0$, there exists an integer $N$ such that for all $i \geq N$, $X_i \cap T$ is, when considered with multiplicity, a $k$-fold branched cover of $\ell \cap T$ along the orthogonal projection map $\mathbb{R}^n \to \ell$, whose angular deviation from $\ell$ is controlled by $\theta$, whose distance to $\ell$ is controlled by $r$. Finally, the number of singular points of $X_i \cap T$ is bounded above by $C(n,k)$. 
    \end{theorem}
    \begin{proof}
        By a rigid motion, we may assume that $\ell$ is the $E_1$-axis. Let $T \subset T_{I,R}$ for $I$ an open interval of finite length and $R > 0$.

        Consider $\theta(n,k)$ as in Theorem \ref{maintechnicaltheorem} and let $\theta' = \min(\theta,\theta(n,k))$, $r' = \min(r,\frac{R}{2}).$
        Apply Theorem \ref{structure-theorem} with $I$, $0 < r' < R$ and $\theta'$ to obtain $N$. Reading off the results from the structure theorem, we have $X_i \cap T_{I,R} \subset T_{I,r'}$, so the distance of $X_i \cap T_{I,R}$ from the $E_1$-axis is bounded by $r$, and the angular deviation of $X_i \cap T_{I,R}$ from the $E_1$-axis is bounded by $\theta$. The last condition (4) tells us that $X_i \cap T_{I,R}$ is a branched $k$-cover over the $E_1$-axis. That implies all of those conditions also hold for $X_i \cap T$, as desired.

        Take $X_i \cap T_{I,R}$ and extend it to an integral stationary geodesic network $X'$ in $\mathbb{R}^n$ by extending the geodesics linearly to infinity. This introduces at most $2{k \choose 2}$ more singular points. From the structure theorem, $X'$ is a $k$-cover, $\theta(n,k)$-flat, and has $|X'_\sing| < \infty$. By Theorem \ref{maintechnicaltheorem}, $|X'_\sing| \leq C(n,k)$. Since $X_i \cap T$ is contained in $X'$, it has at most $C(n,k)$ singular points.
    \end{proof}

\bibliography{references}
\bibliographystyle{amsalpha}

\end{document}